\documentclass[reqno]{amsart}
\usepackage{amssymb,amsmath,amscd,graphicx,color,epstopdf,mathtools,comment}
\usepackage[all]{xy}
\usepackage[all]{xy}
\usepackage{color}
\usepackage{hyperref}
\usepackage{enumerate}
\usepackage{amsthm}
\usepackage{epsfig}
\usepackage[english]{babel}
\usepackage[latin1]{inputenc}

\newcommand{\RR}{{\mathbb{R}}}

\newcommand{\p}{\mathrm{p}}

\renewcommand{\gg}{\mathfrak{g}}
\newcommand{\nn}{\mathfrak{n}}
\newcommand{\hh}{\mathfrak{h}}
\newcommand{\kk}{\mathfrak{k}}
\renewcommand{\ll}{\mathfrak{l}}
\newcommand{\mm}{\mathfrak{m}}
\renewcommand{\aa}{\mathfrak{a}}
\newcommand{\uu}{\mathfrak{u}}
\renewcommand{\tt}{\mathfrak{t}}

\newcommand{\pp}{\mathfrak{p}}
\renewcommand{\sl}{\mathfrak{sl}}
\newcommand{\so}{\mathfrak{so}}
\renewcommand{\sp}{\mathfrak{sp}}
\newcommand{\su}{\mathfrak{su}}
\newcommand{\zz}{\mathfrak{z}}

\newcommand{\cC}{\mathcal{C}}

\newcommand{\cJ}{\mathcal{J}}

\newcommand{\cO}{\mathcal{O}}

\newcommand{\cU}{\mathcal{U}}

\newcommand{\R}{\mathbb{R}}

\newcommand{\Z}{\mathbb{Z}}
\newcommand{\C}{\mathbb{C}}
\newcommand{\NN}{\mathrm{N}}

\newcommand{\Sl}{\mathrm{SL}}
\newcommand{\So}{\mathrm{SO}}
\newcommand{\Sp}{\mathrm{Sp}}
\newcommand{\Su}{\mathrm{SU}}

\newcommand{\K}{\mathrm{K}}
\newcommand{\A}{\mathrm{A}}
\renewcommand{\L}{\mathrm{L}}

\newcommand{\TT}{\mathrm{T}}
\newcommand{\U}{\mathrm{U}}
\newcommand{\G}{\mathrm{G}}
\newcommand{\M}{\mathrm{M}}
\newcommand{\W}{\mathrm{W}}
\renewcommand{\H}{\mathrm{H}}

\renewcommand{\Z}{\mathrm{Z}}

\newtheorem{lemma}{{\bf Lemma}}
\newtheorem{theorem}{{\bf Theorem}}
\newtheorem{corollary}{{\bf Corollary}}

\theoremstyle{remark}
\newtheorem{remark}{Remark}
\newtheorem{example}{Example}

\title[Diagonalization of the Toda flow]{Diagonalization of the Toda flow on conjugacy classes of matrices with simple spectrum}

	 \author{David Mart\'inez Torres}
	 \address{Department of Applied Mathematics, ETSAM Section, Universidad
	 	Polit\'ecnica de Madrid,
	 	Avda. Juan de Herrera 4, 28040 Madrid, Spain}
	 \email{df.mtorres@upm.es}
	 
	  \author{Carlos Tomei}
	 \address{Deapartamento de Matemática, PUC-Rio, R. Mq. S. Vicente 225, 22453-900 Rio de Janeiro, Brazil}
	 \email{carlos.tomei@mat.puc-rio.com}
	 
\begin{document}
\maketitle

\begin{abstract} We construct coordinates on conjugacy classes of traceless complex matrices with simple spectrum that diagonalize the non-periodic Toda vector field.  By this we mean that the coordinates, defined on an open and dense neighborhood of any diagonal matrix in the conjugacy class, decouple the Toda vector field into a sum of multiples of the Euler and rotational vector field in $\C$. Using Lie theoretic methods, we extend this construction from $\sl_\C$ to arbitrary complex semisimple Lie algebras and to their real forms.
 \end{abstract}

\section{Introduction}

The (non periodic) Toda vector field was introduced in \cite{T} as the Hamiltonian vector field of the total energy of a physical system: a chain of $n$ particles in the line connected by $n-1$ springs which exert an exponential restoring force. Flascha's change of variables \cite{F} turned it into a vector field in Lax form on Jacobi matrices, thus making methods from matrix theory
available. One can make sense of the Lax pair expression found by Flaschka in larger matrix spaces, and, therefore, define the Toda vector field in more generality.
Our point of departure is the set of traceless complex matrices, in which  the Toda vector field is given by
\begin{equation}\label{eq:Toda}
X'=[X,\pi_\kk X],\quad \mathrm{I}=\pi_\kk+ \pi_{\uu},\quad X\in \mathfrak{sl}_\C,
\end{equation}
where the identity map $\mathrm{I}$ on $\mathfrak{sl}_\C$ is decomposed as the sum of projections
onto traceless skew-Hermitian matrices $\kk$, and traceless upper triangular matrices with real diagonal entries $\uu$.
 Because the expression \eqref{eq:Toda} uses fundamental vector fields of the action
of $\Sl_\C$ on $\sl_\C$ by conjugation, the Toda vector field is tangent to all conjugacy classes of matrices in $\mathfrak{sl}_\C$.  In particular, the flow is  isospectral: $X(t)$ and $X(0)$ have the same spectrum.

Moser's norming constants \cite{Mo} provided a global parametrization of Jacobi matrices that greatly simplified the expression of Toda vector field: Once the spectrum is fixed, it is transformed  in
the radial projection of a diagonal linear vector field onto the positive sector of the $(n-1)$-sphere. Larger matrix domains for the vector field lead naturally  to the search for analogous coordinates. 

Building on ideas developed in \cite{LST1,MT,LMST}, we  construct such coordinates on dense domains of any conjugacy class  with simple spectrum in the matrix algebra $\mathfrak{sl}_\C$ (Jacobi matrices have simple spectrum).
As it  turns out, our results are valid not just for  $\mathfrak{sl}_\C$, but for arbitrary complex semisimple Lie algebras. 

\medskip
We briefly recall how to define the Toda vector field and the counterparts of conjugacy classes of matrices with simple spectrum in this more general Lie algebra setting.  Equation \eqref{eq:Toda} is tied to the Gram-Schmidt factorization in $\Sl_\C$ associated with the standard Hermitian inner product
\[\Sl_\C=\K\A\NN.\]
Here, the factor $\K$ is the special unitary group,
\[\K=\{k\in \Sl_\C\,|\, kk^*=\mathrm{I}\},\]
$\A$ is the group of determinant 1 diagonal matrices with positive real entries and $\NN$ is the group of upper triangular matrices with diagonal entries equal to 1,
\[\A=\left\{\begin{pmatrix} a_1 &  &   & 0 \\
 & & \cdots &   \\
 0 & &    & a_n
  \end{pmatrix}\in \Sl_\R,\,a_i>0\right\}, \, \NN=\left\{\begin{pmatrix} 1 & n_{12} & &   n_{1n} \\ 0 &
  1 &  & n_{2n} \\
  & & \cdots   &  \\
  0 & &   & 1
  \end{pmatrix}\in \Sl_\C\right\}.\]
The product $\U=\A\NN$  is the group of upper triangular
matrices with positive diagonal coefficients. The Lie algebras of the factors yield the direct sum decomposition
\[\mathfrak{sl}_\C=\mathfrak{k}\oplus \aa\oplus \nn=\kk\oplus\uu \]
on which formula \eqref{eq:Toda} relies.
The Gram-Schmidt algorithm provides an example of an Iwasawa factorization of a complex
semisimple
Lie group $\G=\K\A\NN$. In such a factorization  the group $\K$ is a maximally compact subgroup of $\G$, and it
comes as the fixed-point set of an anti-holomorphic group involution. The  Lie group
$\A$ integrates a maximal abelian subalgebra $\aa$ of the -1 eigenspace of the differential of
the involution
at the identity element. The  subgroup $\NN$ integrates the sum of positive root spaces $\nn$, for a chosen root ordering
of the roots of $(\gg,\aa)$; we denote the sum of negative root spaces by $\ll$.

Once an Iwasawa factorization of $\G$ has been fixed,  the ensuing Iwasawa decomposition of $\gg$ allows to define the Toda vector
field on $\gg$ by the very same formula  \eqref{eq:Toda}. All Iwasawa factorizations of $\G$ are conjugate, so the qualitative behavior of all these vector fields is the same. The Toda vector field on $\gg$ is tangent to all the orbits of the adjoint action of $\G$ on $\gg$, in particular to the regular semisimple adjoint orbits. These are the
adjoint orbits of maximal dimension that meet the Cartan subalgebra $\hh=i\aa\oplus \aa$, and they generalize
conjugacy classes of matrices in $\sl_\C$ with simple spectrum.

\medskip

Our main result is that the Toda vector field on a regular semisimple adjoint orbit is obtained by juxtaposing diagonal linear vector fields defined on covering charts.

\begin{theorem}\label{thm:main} Let  $\cO\subset \gg$ be a regular semisimple adjoint orbit. Then around any $\Lambda\in \cO\cap \hh$ there exist local coordinates with image $\ll\times \ll$
that transform the Toda vector field in the  (complete) diagonal linear  vector field.
  \begin{equation}\label{eq:Toda-coordinates}
 (Y,Z)\in \ll\times\ll,\quad (Y',Z')=([Y,\pi_{i\aa}\Lambda],[Z,-\pi_\aa \Lambda]), \quad \mathrm{I}|_\hh=\pi_{i\aa}+\pi_\aa.
  \end{equation}
Moreover, the domain of each  local coordinate is a dense subset of $\cO$, and
the union of these subsets over the finite set $\cO\cap \hh$ covers $\cO$.
\end{theorem}

The elements in $\cO\cap \hh$ are parametrized by an appropriate Weyl group (see Section \ref{sec:real forms}).
The two factors in $\ll\times \ll$
relate to  the Toda vector field
\[X'=[X,\pi_\kk X]=-[X,\pi_{\uu} X]\]
being tangent to both adjoint $\K$ orbits and adjoint $\U$ orbits.

\medskip
It is instructive to spell out Theorem \ref{thm:main} for $\G=\Sl_\C$. For each diagonal matrix $\Lambda\in \sl_\C$
with simple spectrum, we have coordinates given by strictly lower triangular matrices $Y,Z$ 
that linearize the Toda vector field in a dense open subset of the conjugacy class of $\Lambda$,
\[Y'=[Y,i\Im \Lambda],\quad Z'=[Z,-\Re \Lambda].\]
Here, $\Im \Lambda$ and $\Re \Lambda$ are the imaginary and real parts of the matrix $\Lambda$.
Thus,  phase and amplitude evolution of $X$ split in the variables $Y$ and $Z$, respectively,
and the evolution occurs separately on each entry of $Y$ and $Z$ -- the coordinates diagonalize the vector field.

\medskip
If the diagonal matrix $\Lambda$ has real entries/eigenvalues,
then the Toda vector field is tangent to the submanifold $\cO_0=\cO\cap \sl_\R$ consisting
of real matrices conjugated to $\Lambda$. Denote the strictly lower triangular real matrices by    $\ll_0= \ll\cap \sl_\R$. Then  the restriction of the linear coordinates $Y,Z$ to
$\ll_0\times \ll_0$
are diagonalizing coordinates around $\Lambda$ for the restriction of the Toda vector field to $\cO_0$,
\[(Y,Z)\in \ll_0\times\ll_0,\quad (Y',Z')=(0,[Z,-\Lambda]).\]
These coordinates on $\cO_0$ coincide with the ones described in \cite{LMST}, where they were constructed without using  complex matrices.

We are led to the  following problem: To look for other conditions
on the simple spectrum of a diagonal matrix $\Lambda$,
so that we can find a subgroup  $\G_0\subset \Sl_\C$ such that the Toda vector field is tangent to the
$\G_0$ conjugacy class of $\Lambda$, and the coordinates from Theorem \ref{thm:main} restrict to linearizing coordinates  on the $\G_0$ conjugacy
class (possibly after a suitable modification).

This problem is best analyzed in the setting of real forms of any complex semisimple Lie group $\G$.
Let $\G_0$ be the real form of
 $\G$ given as the fixed-point set of the anti-holomorphic
group involution $\tau$.
Under an alignment condition for $\tau$
and a given Iwasawa factorization 
of $\G$, the groups $\G$ and $\G_0$ have compatible  factorizations
(see Section \ref{sec:real forms} for details). In particular, the Cartan subalgebra $\hh$ of $\gg$ defined by the Iwasawa decomposition has to be invariant by $\tau$ and with fixed-point set a maximally non-compact Cartan subalgebra $\hh_0$ of the Lie algebra $\gg_0$ of $\G_0$.


\begin{theorem}\label{thm:relative} Let $\G$ and $\G_0$ be a complex semisimple Lie group and a real form with compatible Iwasawa factorizations.
 Let $\cO\subset \gg$ be a regular semisimple adjoint orbit with non-empty intersection with the
maximally non-compact Cartan subalgebra $\hh_0\subset \gg_0$, and let $\cO_0$ be the $\G_0$ orbit of one such intersection point.
Then the Toda vector field on $\gg$ has the following properties.
\begin{enumerate}[(i)]
\item It
is tangent to $\cO_0\subset \gg_0$.
\item Around any point in $\cO_0\cap \hh_0$ there exist coordinates on $\cO_0$ that transform the Toda vector field in a  linear  vector field.
\item The domain of these coordinates is an open dense subset of $\cO_0$, and the union of these domains over the finitely many points in $\cO_0\cap \hh_0$ covers $\cO_0$.
\end{enumerate}
\end{theorem}

The linearizing coordinates appearing in Theorem \ref{thm:relative} are not just the restriction of the coordinates
from Theorem \ref{thm:main} to a suitable (real) subalgebra of $\ll\times \ll$. This happens if  $\G_0$ is a split real form of $\G$. For instance, this is the case
for $\Sl_\R\subset \Sl_\C$. In general, the restriction of the second coordinate set $Z$ in \eqref{eq:Toda-coordinates}
has to be modified. The elements in $\cO_0\cap \hh_0$ are parametrized by a suitable Weyl group.

\medskip

Theorem \ref{thm:relative} for $\G=\Sl_\C$ has indeed a reformulation in terms of constraints on the  spectrum of $\Lambda$. These constraints are refinements of the natural requirement that the spectrum  be preserved or reversed under complex conjugation. More precisely, recall that, as well as $\sl_\R$, the other (conjugacy classes of) non-compact real forms of $\sl_\C$
are the quaternion Lie algebra $\sl_\mathbb{H}$ ---block $2\times 2$ complex matrices with each block corresponding to a quaternion--- and the special unitary Lie algebras $\su(p,q)$ ---the isotropy Lie algebras of the the pseudo-Hermitian forms $-z_1\bar{z}_1-\cdots-z_p\bar{z}_p+z_{p+1}\bar{z}_{p+1}+\cdots+z_{p+q}\bar{z}_{p+q}$, $p \ge q$ \cite[p. 444-445]{He}. 
\begin{theorem}\label{thm:slC} Let $\Lambda \in \sl_\C$ be a diagonal matrix with simple spectrum that satisfies
either of the following constraints:
\begin{enumerate}[(i)]
 \item $\Lambda\in \sl_\R$;
 \item $\Lambda\in \sl_\mathbb{H}$, i.e., $\Lambda=\begin{pmatrix} \lambda_1 & 0 & &  &  \\ 0 &
 \overline {\lambda}_1 &  & & \\
 & & \cdots &  & \\
  & & & \lambda_n & 0 \\
  & &  & 0 & \overline{\lambda}_n
  \end{pmatrix}$;
\item $\Lambda=\begin{pmatrix} \Lambda_{q} & 0 & 0   \\ 0 & \Im_{p-q} & 0\\
 0 & 0 & -\overline{\Lambda}_q
  \end{pmatrix}$,
where $\Lambda_q$ and $\Im_{p-q}$ are diagonal matrices of size $q$ and  $p-q$
of complex and purely imaginary complex numbers, respectively.
\end{enumerate}
Then the Toda vector field on $\sl_\C$ is tangent
to the $\G_0$ conjugacy class of $\Lambda$, and
there exists an atlas of the conjugacy class that transforms the Toda vector field
in a finite collection of  linear vector fields, where
the real form $\G_0$ that corresponds respectively to each constraint is 
\begin{enumerate}[(i)]
  \item $\Sl_\R$;
  \item $\Sl_\mathbb{H}$;
  \item $c\mathrm{SU}(p,q)c^{-1}$,
  where
  \[c=g\begin{pmatrix}I_{p-q} & 0 & 0\\
    0 & \tfrac{1}{\sqrt{2}}\mathrm{I}_q & \tfrac{1}{\sqrt{2}}\mathrm{I}^{\mathrm{op}} \\
    0 & -\tfrac{1}{\sqrt{2}}\mathrm{I}^{\mathrm{op}} &  \tfrac{1}{\sqrt{2}}\mathrm{I}_q
    \end{pmatrix},\quad \mathrm{I}^\mathrm{op}=\begin{pmatrix}
    0 &\cdots & 1\\
    & \cdots &\\
    1 &\cdots & 0
    \end{pmatrix},
\]
and $g$ is the permutation matrix given by
\[ (1, 2, 3, \ldots, p+q) \mapsto (1+q,2+q,3+q,\cdots ,p, 1,2,3,\cdots, q) \ . \] 
\end{enumerate}
\end{theorem}

\medskip

Finally, we collect some of the cases  of Theorems \eqref{thm:main} and \eqref{thm:slC} in a result that reflects the title of the paper.

\begin{corollary}\label{cor:simple}
Let $\cO$ be a conjugacy class of diagonalizable matrices in
\[\sl_\mathbb{K},\,\mathbb{K}=\R,\C,\mathbb{H}, \quad \so_\C,\quad \sp_\C\]
whose spectrum is simple and belongs to $\mathbb{K}$. Then
there exists an atlas of $\cO$ that transforms the Toda vector field
in a finite collection of  diagonal linear  vector fields.
\end{corollary}

\medskip

The Hamiltonian genealogy of the original Toda vector field  naturally led to the study of its integrability (in the sense of Liouville). 
The standard symplectic structure in $\RR^{2n}$, associated with the dynamics of $n$ particles in the line, yields a symplectic structure on $\cJ$, the set of traceless Jacobi matrices, after getting rid of the simple dynamics of the center of mass, by the  change of variable suggested by Flaschka \cite{F}.
The eigenvalues along an integral curve $J(t) \in \cJ$  are preserved, and Poisson commute: they  provide the required conserved quantities to yield the complete integrability of the system on $\cJ$.

Larger phase spaces for the Toda flow raised new  integrability issues. 
Adler \cite{Adler} showed that $\cJ$ is a coadjoint orbit of the group of upper triangular matrices. The integrability on regular orbits in $\sl_\RR$ (generic, real symmetric matrices with simple spectrum) was proved in \cite{DLNT1}. Subsequently,
 the symmetry requirement was dropped \cite{DLT2} and 
 the Lie algebra $\sl_\R$ was replaced by an arbitrary real semisimple Lie algebra (for the symmetric case) \cite{GS}. For the Toda vector field in $\cJ$, Agrotis et al. \cite{Agrotis} proved another form of integrability: More conserved quantities are obtained at the expense of less commutativity of its Hamiltonians.

\medskip
On the other hand, conserved quantities should be naturally derived from the fact that the Toda vector field is tangent  to upper triangular and orthogonal orbits by conjugation (from Remark \ref{hierarchy}, this intersection collects the integral curves of the entire Toda hierarchy). This more Lie algebraic approach does not account for the symplectic geometry, in particular the possible commutativity of the Hamiltonians. In exchange it provides the local coordinates which, as in Theorem \ref{thm:main}, diagonalize the vector field.

\medskip

{\bf Acknowledgments.} The authors thank Ricardo Leite and Nicolau Saldanha for extensive conversations. DMT and CT acknowledge  finantial support of CNPQ and FAPERJ. DMT and CT also acknowledge respectively partial finantial support  of Universidad Polit\'ecnica de Madrid, Programa Propio 2024, and of StoneLab.

\section{The case of  traceless complex matrices}\label{sec:cx}
In this section we 
provide a    global description of the Toda flow on conjugacy classes of matrices in
 $\sl_\C$ with simple spectrum. 
 
 Our methods 
 are not the usual ones from matrix theory, but rather pertain to Lie theory. We do this in order to prepare the reader for the more general arguments in the coming sections, valid for
arbitrary complex semisimple Lie algebras.

\medskip
\subsection{Factorizations and coordinates on conjugacy classes}\label{sec:factorizations}
For any conjugacy class of matrices in $\sl_\C$
 with simple spectrum, we  construct suitable local coordinates centered at a diagonal matrix. To do that we  combine properties of the $\mathrm{Q}\mathrm{R}$ and the  $\L\U$
factorization of matrices in $\Sl_\C$. We  use uppercase letters for Lie algebra elements  and  lowercase letters for
Lie group elements.

\medskip

The Gram-Schmidt algorithm factors a  matrix in $\Sl_\C$
into a special unitary matrix, a diagonal matrix with strictly positive entries, and an
upper triangular matrix with diagonal entries equal to 1,
\[\Sl_\C=\K\A\NN.\]
Multiplying  $\U=\A\NN$ we obtain the classical $\mathrm{QR}$ factorization
\[\Sl_\C=\K\U,\]
where $\U$ stands for triangular matrices with strictly positive diagonal entries. Let $\L$ be the group of lower triangular matrices with diagonal entries equal to 1.

As a preliminary step,  we discuss from a geometric viewpoint certain subsets of matrices with relevant factorization properties.

\medskip

  Let $\cC$ be the subset of  unitary matrices in $\Sl_\C$ that  possess an $\L\U$ factorization,
\[\cC=\{k\in \K\,|\,\exists l\in \L,u\in \U,\, k=lu\}.\]
Clearly, any unitary matrix belongs to $\cC$ exactly when all its principal
minors are positive. Since the conditions to produce an $\L\U$ and a $\U\L$ factorization
are the same, the subset $\cC$ is invariant under inversion. There is a more useful  characterization of $\cC$: It is the image of the restriction
to $\L$ of the first projection in the $\mathrm{QR}$ factorization $\kappa:\Sl_\C\to \K$,
\begin{equation*}\kappa:\L{\longrightarrow}{\cC}.
 \end{equation*}
From a geometric viewpoint, $\cC$ is an embedded submanifold of $\K$ and the projection is a real analytic diffeomorphism.

A similar  analysis applies to
the restriction to $\L$ of the second projection in the $\mathrm{QR}$ factorization
 $\nu:\Sl_\C\to \U$.
Define
\begin{equation}\label{eq:more-LU}
\mathcal{D}=\nu(\L)=\{u\in U\,|\,\exists l\in \L,k\in \K,\, l=ku\}.
\end{equation}

\begin{lemma}\label{lem:more-LU} The subset $\mathcal{D}$ is a closed embedded submanifold of $\U$ and
 $\nu:\L\to \mathcal{D}\subset \U$
 is a real analytic diffeomorphism onto its image.
\end{lemma}
\begin{proof} The result of applying the (Cartan) involution
\[\theta:\Sl_\C\to \Sl_\C,\quad g\mapsto {(g^*)}^{-1}\]
to the Gram-Schmidt factorization produces the
opposite Gram-Schmidt factorization
$\Sl_\C=\K\A{\L}$. Switching the last
two factors  gives  yet another factorization $\Sl_\C=\K\L\A$. Denote by $\varpi:\U\to \A$ the restriction to $\U$ of the third projection of the latter factorization.
It follows from \eqref{eq:more-LU} that $\mathcal{D}=\varpi^{-1}(e)$, where $e$ is the identity element. We claim that $e$ is a regular value
for $\varpi$: If $\varpi(u)=e$, then
\[u=kl,\quad k\in \K,l\in \L.\]
Since for any $a\in \A$ we have
\[ua\subset \U,\quad \varpi(ua)=a,\]
the differential of $\varpi$ at $u$ is surjective: $\mathcal{D}$ is a closed embedded submanifold of $\U$.

Equation \eqref{eq:more-LU} implies that the composition of the restriction $\nu:\L\to \U$ with the projection
$\Sl_\C=\K\A\L\to \L$ equals the identity on $\L$.
Therefore $\nu:\L\to \U$
is an injective immersion onto its image $\mathcal{D}$. Since $\mathcal{D}$ is an
embedded submanifold, this implies
that $\nu$ is real analytic diffeomorphism from $\L$ to $\mathcal{D}$.
\end{proof}

\begin{remark} The description of $\mathcal{D}$ as the fiber $\varpi^{-1}(e)$ provides the following  alternative characterization:
 it is the subset of matrices $u\in \U$ such that all principal minors of ${(u^*u)}^{-1}$  equal 1.
 However, our conceptual viewpoint does not pursue characterizing solutions of certain factorization problems in terms of matrix features (e.g., the behavior of minors). It focusses on the
 existence of those solution sets together with some of their key geometric properties.
 This viewpoint lends itself to Lie theoretic generalizations.
\end{remark}

\begin{remark}\label{rem:more-LU}
Lemma \ref{lem:more-LU} puts on equal footing the three manifolds involved in
the $\L\U$ factorization for matrices in the special unitary group $\K\subset \Sl_\C$.
\begin{equation*}
\xymatrix{
 &  & \L \ar[dll]^{\kappa}_{\cong}\ar[drr]_{\nu}^\cong  & &  \\
\cC\ar[rrrr]^\cong &  & & & \mathcal{D},}\qquad \xymatrix{
 &  & l=ku\ar@{|->}[dll]\ar@{|->}[drr] & &  \\
k \ar@{|->}[rrrr] &  & & & u}
\end{equation*}
\end{remark}

\bigskip
We now use the three submanifolds in Remark \ref{rem:more-LU} to parametrize open dense subsets  of a conjugacy class  $\cO\subset \sl_\C$ of
matrices with simple spectrum. To do that we  contextualize diagonal matrices in a Lie theoretic setting.

Let $\hh\subset \sl_\C$ denote the maximal abelian subalgebra  of  diagonal matrices
\[\hh=\left\{\begin{pmatrix} \lambda_1 &  &   & 0 \\
 & & \cdots &   \\
 0 & &    & \lambda_n
  \end{pmatrix}\in \sl_\C\right\}.\]
Let $\Lambda\in \cO$ be a diagonal matrix, i.e., $\Lambda\in \cO\cap \hh$. That
$\Lambda$ have simple spectrum is equivalent to the centralizer
\[\Z(\Lambda)=\{g\in \Sl_\C\,|\, g\Lambda g^{-1}=\Lambda\}\]
being equal to the maximally abelian subgroup
 $\H\subset \Sl_\C$ of  diagonal  matrices  (the group whose Lie algebra is $\hh$).
 
In  the rest of the paper, an invertible matrix/group element as superscript  denotes conjugation by that matrix/group element, $X^g = g X g^{-1}$. Also, $\ll$  denotes the strictly lower triangular matrices (the Lie algebra of $\L$).

\begin{lemma}\label{lem:LU-maps} The maps
\begin{equation}\label{eq:LU-maps}\Phi_\cC: \cC\times \L\to \cO,\quad (k,l)\mapsto \Lambda^{(lk)^{-1}},\quad
\Phi_\mathcal{D}:\mathcal{D}\times \L\to \cO,\quad (u,l)\mapsto \Lambda^{ul^{-1}},
\end{equation}
have the following properties.
\begin{enumerate}[(i)]
 \item They are diffeomorphisms onto the same open dense subset $\cU\subset \cO$;
 \item They produce rulings of $\cU$ by
$\cC$ and $\mathcal{D}$ conjugates of the affine subspace
\begin{equation}\label{eq:LU-rulings}
\Lambda+\ll=\Phi_\cC(\{e\}\times \L)=\Phi_\mathcal{D}(\{e\}\times \L).
 \end{equation}
\end{enumerate}
\end{lemma}

A ruling is a fibration of a manifold by affine spaces.

\begin{proof}
Let $\H=\TT\A$, where $\TT$ is a diagonal of phases, and $\A$ is positive diagonal. 
The product map 
\[\cC\times \TT\to \cC\TT\subset \K\] is a diffeomorphism
onto a  dense subset of $\K$ (its complement being those  unitary matrices with at least one trivial principal minor).
Therefore
\[\cC\TT\L\A\subset \Sl_\C\,(=\K\L\A)\]
is a dense open subset. As the compact torus $\TT$ normalizes $\L$, this open dense subset admits the presentation
\[\cC\L\TT\A=\cC\L\H.\]
Standard properties of actions of groups on manifolds imply that
\[\cU=\{\Lambda^g,\,|\,g\in \cC\L\H\}\]
is an open dense neighborhood of $\Lambda$ in $\cO$.
Because $\H=\mathrm{Z}(\Lambda)$  and $\cC$ is invariant under
inversion, the subset $\cU$ equals the image of the map $\Phi_\cC$ in \eqref{eq:LU-maps}.
Moreover, the factorization $\cC\L\H$ as the product of  the (embedded) submanifold   $\cC\L$ and the centralizer of $\Lambda$ implies that $\Phi_\cC$ is a diffeomorphism onto $\cU$. This proves (i) for $\Phi_\cC$.

To prove (i) for $\Phi_\mathcal{D}$ we  factor it through $\Phi_\cC$.
The proof of Lemma \ref{lem:more-LU} (see also \eqref{rem:more-LU})
shows that
\[\kappa\circ \nu^{-1}:\mathcal{D}\to \cC,\quad \nu^{-1}:\mathcal{D}\to\L\]
are diffeomorphisms.
Therefore
\[\mathcal{D}\times \L\to \mathcal{C}\times \L,\quad
(u,l)\mapsto (\kappa\circ \nu^{-1}(u),l{(\nu^{-1}(u))}^{-1})=(k,ll_1^{-1}),
\quad l_1=ku,\]
is a diffeomorphism that followed by $\Phi_\mathcal{C}$ gives a factorization of $\Phi_\mathcal{D}$.
Hence $\Phi_\mathcal{D}$ is also a diffeomorphism onto $\cU$. This finishes the proof of (i).

To prove (ii) we need to show that the $\L$ conjugacy class of $\Lambda$ equals the affine subspace $\Lambda+\ll$  of lower triangular matrices whose diagonal is $\Lambda$. Observe that since $\Lambda$ is a diagonal matrix, the inclusion of the $\L$ conjugacy class of $\Lambda$ in $\Lambda+\ll$ is clear.
To check the equality
we first use that the unipotent group  $\L$ acts freely on any diagonalizable
matrix with simple spectrum. The reason is that the isotropy group for the conjugation action of $\Sl_\C$
of any such matrix is made of diagonalizable matrices (semisimple
elements). Second, we observe that $\Lambda+\ll$ is invariant under the action of  $\L$ by conjugation, and, since all matrices there are diagonalizable with simple
spectrum, $\L$ acts freely on $\Lambda+\ll$. Third, by a dimension count that uses the triviality of the isotropy group  of the action of $\L$, all $\L$ conjugacy classes are open in $\Lambda+\ll$.  Finally, since $\Lambda+\ll$ is connected it has to agree with the
$\L$ conjugacy class of any of its elements, in particular with that of $\Lambda$.
\end{proof}

We use the previous rulings to construct a map from $\cU$ to $(\Lambda+\ll)\times (\Lambda+\ll)$ by first
inverting $\Phi_\cC$ and $\Phi_\mathcal{D}$, next projecting onto $\L$,
and finally applying the identification \eqref{eq:LU-rulings}
of $\L$ with $\Lambda+\ll$,

\begin{eqnarray}\label{eq:local-coordinates-LU}
 \varphi=(\Phi_\cC\circ \mathrm{pr}_2\circ \Phi_\cC^{-1},\Phi_\mathcal{D}\circ \mathrm{pr}_2\circ \Phi_\mathcal{D}^{-1}): \cU &
 \to & (\Lambda+\ll)\times (\Lambda+\ll) \nonumber \\
 X=\Lambda^{(l_1k)^{-1}}=\Lambda^{ul_2^{-1}}&\mapsto & (Y=\Lambda^{l_1^{-1}},Z=\Lambda^{l_2^{-1}})
\end{eqnarray}

\begin{corollary}\label{cor:local-coordinates-LU} The map $\varphi:\cU\to (\Lambda+\ll)\times (\Lambda+\ll)$ defined
in \eqref{eq:local-coordinates-LU} is a real analytic
diffeomorphism whose inverse is given by
\[ (Y=\Lambda^{l_1^{-1}},Z=\Lambda^{l_2^{-1}}) \mapsto Y^{k^{-1}}=Z^{u}, \quad \hbox{for} \  \ l_1^{-1}l_2=ku \ . \]
Hence,  the $\cC$ and $\mathcal{D}$ conjugates of the affine subspace $\Lambda+\ll\subset \cU$ give a product structure on $\cU$.
\end{corollary}

We refer to the above construction as the {\bf $(Y,Z)$ variables} on $\cU$.

\begin{proof}
 To show that $\varphi$ is a diffeomorphism we construct its inverse map: Given $(Y,Z)\in (\Lambda+\ll)\times (\Lambda+\ll)$,
we first apply the inverses $(\Phi_\cC^{-1},\Phi_\mathcal{D}^{-1})$  of the diffeomorphisms in \eqref{eq:LU-rulings}
   to write $Y=\Lambda^{l_1^{-1}},Z=\Lambda^{l_2^{-1}}$.
Second,  we factor
\[l_1^{-1}l_2=ku,\quad  k\in \cC,\,u\in \mathcal{D}.\]
By the definition of $\varphi$, the inverse $\varphi^{-1}$ is as given in the statement.

The real analyticity of the diffeomorphism is a consequence of the real analyticity of the  structural maps of a  Lie group, of its different factorizations (Iwasawa, $\L\U$), and of its linear actions (more specifically, of the conjugation action).
\end{proof}

\begin{remark} The $(Y,Z)$ variables for real matrices with real, simple spectrum were already introduced in \cite{LMST}. For a phrasing of the formulas above in standard matrix notation, see its Introduction.
	\end{remark}

\begin{example}\label{ex:1}
  We  compute explicitly $\varphi^{-1}:(\Lambda+\ll)\times (\Lambda+\ll)\to\cU\subset \cO$  for $\sl_\C(2)$, i.e., the matrix in $\cU$ associated with given $(Y,Z)$ variables. 
 The domain of the chart is
\[ (\Lambda+\ll)\times (\Lambda+\ll)= \left\{(\Lambda+Y,\Lambda+Z)=\left(\begin{pmatrix}
           \lambda & 0\\ y & -\lambda
          \end{pmatrix},\begin{pmatrix}
           \lambda & 0\\ z & -\lambda
          \end{pmatrix}\right)\,|\,y,z\in \C \right\},\]
where $\lambda\in \C\setminus 0$ is fixed.
First, we compute the diffeomorphism
\[(\Lambda+\ll)\times (\Lambda+\ll)\to \L\times \L,\quad  (\Lambda+Y,\Lambda+Z)
\mapsto (l_1,l_2),\]
determined by the equality $(\Lambda^{l_1^{-1}},\Lambda^{l_2^{-1}})=(\Lambda+Y,\Lambda+Z)$:
\[l_1(Y)=\begin{pmatrix}
           1 & 0\\  -\tfrac{y}{2\lambda} & 1
          \end{pmatrix}, \quad l_2(Z)=\begin{pmatrix}
           1 & 0\\  -\tfrac{z}{2\lambda} & 1
          \end{pmatrix}.\]
Second, for $l\in \L$ we compute its unitary factor
in the Gram-Schmidt factorization:
\[\kappa(l)=\begin{pmatrix}
           \tfrac{1}{r} &  -\tfrac{\overline{s}}{r}\\  \tfrac{s}{r} & \tfrac{1}{r}
          \end{pmatrix},\quad l=\begin{pmatrix}
           1 & 0\\  s & 1
          \end{pmatrix}, \quad r=\sqrt{1+s\overline{s}}.\]
Third, we conjugate $\Lambda+Y$ by $\kappa(l)^{-1}=\kappa(l)^T$ for
$l=l_1^{-1}(Y)l_2(Z)$:
\begin{eqnarray*}
 \varphi^{-1}: (\Lambda+\ll)\times (\Lambda+\ll)\cong \C\times \C &\longrightarrow & \cO\subset \sl_\C(2)\\
       \left(\begin{pmatrix}
           \lambda & 0\\ y & -\lambda
          \end{pmatrix},\begin{pmatrix}
           \lambda & 0\\ z & -\lambda
          \end{pmatrix}\right)  & \longmapsto & \frac{1}{r^2}\begin{pmatrix}
           \lambda r^2+\overline{s}z & \overline{s}(2r^2\lambda+\overline{s}z)\\
         z & -\lambda r^2-\overline{s}z
          \end{pmatrix},\quad s=\tfrac{y-z}{2\lambda}.
\end{eqnarray*}
Finally, upon substitution of $r$ and $s$ as functions of $y,z$, we obtain
\begin{align*}
&\tfrac{1}{1+\tfrac{|y-z|^2 }{4|\lambda|^2}}
\begin{pmatrix}  \lambda\left(1+\tfrac{|y-z|^2}{4|\lambda|^2}\right)
+\tfrac{(\overline{y}-\overline{z})z}{2\overline{\lambda}}
& \tfrac{(\overline{y}-\overline{z})}{2\overline{\lambda}}\left(
  2\lambda\left(1+\tfrac{|y-z|^2}{4|\lambda|^2}\right)
 +
\tfrac{(\overline{y}-\overline{z})z}{2\overline{\lambda}}\right)   \\
 z & -\lambda\left(1+\tfrac{|y-z|^2}{4|\lambda|^2}\right)
-\tfrac{(\overline{y}-\overline{z})z}{2\overline{\lambda}}
 \end{pmatrix} \\
 &=\tfrac{1}{4|\lambda|^2+|y-z|^2}
\begin{pmatrix}
4\lambda|\lambda|^2+\lambda(\overline{y}-\overline{z})(y+z)
&  8\lambda|\lambda|^2+\tfrac{\lambda(\overline{y}-\overline{z})^2(y+z))}{2\overline{\lambda}}\\
 {4|\lambda|^2z} &
 -4\lambda|\lambda|^2-\lambda(\overline{y}-\overline{z})(y+z))
\end{pmatrix}.
\end{align*}
\end{example}

\subsection{Diagonalization of the Toda vector field}

In this section we  prove that, in
$(Y,Z)$ variables,  the Toda vector field is a (complete)
diagonal linear  vector field. This is the fundamental ingredient in the proof of Theorem \ref{thm:main} for $\sl_\C$.

\medskip

The Toda vector field \eqref{eq:Toda} on $\sl_\C$ is given by
\[X'=[X,\pi_\kk X],\quad \mathrm{I}=\pi_\kk+\pi_\uu.\]
A well known result of Symes  describes its solutions by means of the $\mathrm{Q}\mathrm{R}$ factorization in $\Sl_\C$ \cite{Sy}:
The trajectory  starting at any $X\in \sl_\C$ is given by
 \begin{equation}\label{eq:Toda-solution-Iwasawa-sl}
 X^{\kappa(\exp(tX))^{-1}}=X^{\nu(\exp(tX))}.
 \end{equation}
Whereas this allows to compute every trajectory, it may fail short from describing certain qualitative properties of families of trajectories as $X$ varies. 
The following  factorization result gives full control of the trajectories of the Toda flow
starting at any point in the domain of our local coordinates.

\begin{theorem}\label{pro:Toda-solution-complex}
 Let $\cO\subset \sl_\C$ be a conjugacy class of complex matrices with simple spectrum, let $\cU\subset \cO$
 be the open subset from Lemma \ref{lem:LU-maps} centered at the diagonal matrix $\Lambda\in \cO$, and
 let $X\in \cU$, \[X=(\Lambda+Y)^{k^{-1}}=(\Lambda+Z)^u,\quad  Y,Z\in \ll,\, k\in \cC,\,u\in \mathcal{D}.\]
 Then the trajectory
 of the Toda vector field starting at $\Lambda$ admits the factorizations
 \begin{equation}\label{eq:Toda-factorization}
 X(t)=
 {\left((\Lambda+Y)^{\exp(-ti\Im\Lambda)}\right)}^{\kappa(l_\Im(t))^{-1}},\quad l_\Im(t)\in \L,
 \end{equation}
  \begin{equation}\label{eq:Toda-factorization2}
 X(t)=
 {\left((\Lambda+Z)^{\exp(t\Re\Lambda)}\right)}^{\nu(l_\Re(t))},\quad l_\Re(t)\in \L,
 \end{equation}
where $\Re\Lambda$ and $\Im\Lambda$ are the real and imaginary parts of $\Lambda$, respectively.
In particular, the entire Toda trajectory is contained in $\cU$.
\end{theorem}
\begin{proof}
  By Symes' formula,  the trajectory is the conjugation of $(\Lambda+Y)^{k^{-1}}$ by
  \[\kappa(\exp(t(\Lambda+Y)^{k^{-1}}))^{-1},\]
  or, equivalently,
  \begin{equation}\label{eq:Toda-in-chart}
  X(t)=(\Lambda+Y)^{\kappa(k\exp(t(\Lambda+Y)^{k^{-1}}))^{-1}}.
   \end{equation}
The group of diagonal matrices $\H$ normalizes the  group $\L$. This allows to relate
 the exponential of the sum and the product of the exponentials of the summands,
\[\exp(t(\Lambda+Y))=\exp(t\Lambda)l_1(t),\quad l_1(t)\in \L.\]
Therefore, if we factor $k=l_2u_2$, then  can write
\begin{align*}
k\exp(t(\Lambda+Y)^{k^{-1}}) &= \exp(t(\Lambda+Y))k=\exp(t\Lambda)l_1(t)l_2u_2=\\
&=\exp(ti\Im\Lambda)\exp(t\Re\Lambda)l_1(t)l_2u_2=
\exp(ti\Im\Lambda)l_\Im(t)u_3(t),
\end{align*}
where the last equality uses that $\A$ normalizes $\L$ and is contained in $\U$. Combining with
\eqref{eq:Toda-in-chart},
 the trajectory of the Toda vector field starting at $X$ can be written as
\begin{align*}
X(t)&=(\Lambda+Y)^{\kappa\left(k\exp(t(\Lambda+Y)^{k^{-1}})\right)^{-1}}=
(\Lambda+Y)^{\kappa\left(\exp(ti\Im\Lambda)l_\Im(t)u_3(t)\right)^{-1}}=\\
&=\left((\Lambda+Y)^{\exp(-ti\Im\Lambda)}\right)^{\kappa(l_\Im(t))^{-1}},
\end{align*}
and hence \eqref{eq:Toda-factorization} holds.

For the second factorization, we write
\[\exp(t(\Lambda+Z))=l_3(t)\exp(t\Lambda),\quad l_3(t)\in \L.\]
Therefore, if we factor $u=k_2l_4$, then using again that $\H$ normalizes ${\L}$ we write
\begin{align*}
\exp(t(\Lambda+Z)^{u}u) &= u\exp(t(\Lambda+Z))=k_2l_4l_3(t)\exp(ti\Im\Lambda)\exp(t\Re\Lambda)=\\
&=k_2\exp(ti\Im\Lambda)l_\Re(t)\exp(t\Re\Lambda)=
k_3(t)l_\Re(t)\exp(t\Re\Lambda),
\end{align*}
so the Toda trajectory also equals
\[
X(t)=
(\Lambda+Z)^{\nu(k_3(t)l_\Re(t)\exp(t\Re\Lambda))}
=\left((\Lambda+Z)^{\exp(t\Re\Lambda)}\right)^{\nu(l_\Re(t))}.
\]

The various factorizations involved exist for all time,
and this shows that the Toda trajectory is confined in the open dense subset $\cU$.
\end{proof}

The following  result is an immediate consequence of Theorem \ref{pro:Toda-solution-complex}.

\begin{theorem}\label{thm:linear-coordinates}
 Let $\cO\subset \sl_\C$ be a conjugacy class of  matrices with simple spectrum. Then in $(Y,Z)$ variables centered  at a diagonal matrix
 $\Lambda\in \cO$,
 the Toda vector field is 
\begin{equation*}
(Y',Z')=\left([Y,i\Im\Lambda],[Z,-\Re\Lambda]\right).
\end{equation*}
This linear vector field
\begin{enumerate}[(i)]
 \item is defined in the entire affine space $(\Lambda+\ll) \times (\Lambda+\ll)$;
  \item has trajectories whose first projection has compact closure and whose second
 projection is compact if and only if it is stationary;
 \item decouples in the strictly lower triangular entries  $Y_{jk},Z_{\ell m}$ of  $Y$ and $Z$.

 \end{enumerate}
\end{theorem}
\begin{proof}
 The factorizations  in  \eqref{eq:Toda-factorization} and \eqref{eq:Toda-factorization2} of the
 solution $X(t)$ starting at  $X$ imply
 that the Toda vector field projects by
 \[\mathrm{pr}_1\circ \varphi:\cU\to \Lambda+\ll,\quad \mathrm{pr}_2\circ \varphi:\cU\to \Lambda+\ll\]
 onto the vector fields
 \[Y'=[Y,i\Im\Lambda],\quad Z'=[Z,-\Re\Lambda],\]
respectively, proving the formula for the Toda vector field in $(Y,Z)$ variables.

Items (i), (ii) and (iii) are consequences of the specific linear vector field we obtain:
The evolution on each individual entry is 
\[Y'_{jk}=i\Im(\lambda_{k}-\lambda_{j})Y_{jk},\quad Z'_{\ell m}=\Re(\lambda_{\ell}-\lambda_{m})Z_{\ell m},\quad \Lambda=\begin{pmatrix} \lambda_1 &  &   & 0 \\
 & & \cdots &   \\
 0 & &    & \lambda_n
  \end{pmatrix} \ .\]
\end{proof}

\begin{proof}[Proof of Theorem \ref{thm:main} for $\sl_\C$]
The local coordinates in the statement of Theorem \ref{thm:main} are the $(Y,Z)$ variables 
in Theorem \ref{thm:linear-coordinates}. Notice that in the latter theorem
the variables take values in a product of
 affine spaces with a marked point, and, hence,  vector spaces, which is the notation that appears
in Theorem \ref{thm:main}.

By Lemma \ref{lem:LU-maps} the domain of the $(Y,Z)$ variables is dense in $\cO$. We claim that $\Lambda$ is the only diagonal matrix in their domain. To conjugate $\Lambda$ to a different diagonal matrix 
we can only use elements in $w\H$, where $w$ is a  nontrivial permutation matrix divided by its determinant. Since $\L\cap \A=\{e\}$,
\[w\H\cap \cC\L=w\TT\A\cap \cC\L=w\TT\cap \cC,\]
and the latter intersection is empty because a nontrivial permutation matrix has at least one vanishing principal minor. 

It only remains to prove that the union of the domains of the coordinates centered at each diagonal matrix 
cover the conjugacy class $\cO$. Indeed, every  unitary matrix has an $\L\U$ factorization after pivoting. In Lie theoretic language, it is a consequence of the existence of the Bruhat cover of $\Sl_\C$, that we now recall.
Fix  $\Lambda\in\cO\cap \hh$, so that we obtain a bijection
 \[\cO\cap \hh=\bigcup_{w\in \W} \Lambda^w,\]
where $\W\subset \K$ is the permutation group.
The Gram-Schmidt factorization induces an open cover by right
translates by (normalized) permutation matrices of the open dense  Bruhat cell ---defined as the product of upper times lower triangular matrices,
\[\Sl_\C=\bigcup_{w\in \W}\H\NN\H\L w.\]
Therefore, upon acting by conjugation on $\Lambda$, it gives a cover of $\cO$ by open and dense subsets. 
We show now that this  cover is the one by $(Y,Z)$ variables centered at different diagonal matrices in $\cO$.
We  restrict the factorization $\Sl_\C=\K\L\A$ to $\NN$ to rewrite
$\NN\L\A=\cC{\L}\A$,
and use that $\H$ normalizes $\NN$ and $\L$ to rewrite
$\H\NN\H\L=\cC\L\H$.
The outcome is
 \begin{equation*}
\cO=\bigcup_{w\in \W}\Lambda^{\H\NN\H\L w}=\bigcup_{w\in \W}{(\Lambda^w)}^{\cC\L\H}=
\bigcup_{w\in \W}{(\Lambda^w)}^{\cC\L},
  \end{equation*}
that is, $\cO$ is the union of the domains of the $(Y,Z)$ variables.
\end{proof}

\begin{remark}
The $(Y,Z)$ variables  provide diagonal evolution for the vector fields in the so called Toda hierarchy 
\[ X ' = [ X, \pi_\kk f(X) ], \]
where $f$ is an arbitrary fixed polynomial (or, more generally, an entire function). More explicitly,
\begin{equation*}
	(Y',Z')=\left([Y,i\Im\ f(\Lambda)],[Z,-\Re\ f(\Lambda)]\right).
\end{equation*}
Indeed, Symes' factorization formula still holds,
\begin{equation} \label{Symes}
	X^{\kappa(\exp(tf(X)))^{-1}}=X^{\nu(\exp(tf(X)))},
\end{equation}
and the arguments in the proofs of Theorem \ref{pro:Toda-solution-complex} and \ref{thm:linear-coordinates}  are still valid, as polynomials of matrices commute with conjugation. 
\end{remark}

\begin{remark} \label{hierarchy}  Symes' factorization formula for the entire hierarchy comes up naturally from the study of the intersection of upper triangular and orthogonal conjugacy classes. A matrix in the intersection $X^{\U} \cap X^{\K}$ gives rise to elements $k\in \K$, $u\in \U$ such that $X$ and $ku$ commute, and, therefore  for a function $h$, 
\[ku=h(X) .\]
An easy computation shows that, indeed, taking $h = \exp(t f)$, one recovers  the
 different representations of the solution of $X ' = [ X, \pi_\kk f(X) ]$ in equation \eqref{Symes}.
\end{remark}

\begin{remark}
The $(Y,Z)$ variables are  compatible with profiles, defined in \cite[Section 3.3]{LMST}, which generalize vector spaces of Hessenberg and band matrices. 
Given a set of pairs $ S = \{ (i,j) , \ i \ge j, \  i, j \in \{1, 2, \ldots,n\} \}$, the  profile $\p$ generated by $S$ is
\[ \p = \{ (i,j) \ | \ \exists \  (i', j') \in S , \ i \le i', j \ge j'\} \, .\]
Let $V_\p \subset \sl_\C$ be the subspace spanned by the matrices $E_{ij} = e_i \otimes e_j$ for $(i,j) \in \p$. It turns out that the Toda vector field is tangent to $V_\p$, and that $\cO\cap V_\p$ is a submanifold, provided the conjugacy class $\cO$ has simple spectrum. Moreover, $(Y,Z)$ variables restrict to profiles yielding a diffeomorphism 
\[  \varphi : \cU \cap V_\p\to (\Lambda + \ll) \times (\Lambda + \ll\cap V_\p)   \, , \quad X \mapsto ( Y, Z)   \] 
\end{remark}

\begin{example} We illustrate Theorem \ref{thm:main}
for the chart $\varphi^{-1}$ of $\cO\subset \sl_\C(2)$ computed in Example 1. The expression for the Toda vector field is
\begin{equation*} X'=
\begin{pmatrix} c(b+\overline{c}) & -2(a\overline{c}+bi\Im a) \\ -2c\Re a & -c(b+\overline{c})\end{pmatrix},
\quad X=\begin{pmatrix} a & b \\ c & -a\end{pmatrix}.
\end{equation*}
According to Example 1, the chart of $\cO$, to which the Toda vector field is tangent,
is given in terms of the off-diagonal entries of the $(Y,Z)$ variables by
\begin{eqnarray*}
 \varphi^{-1}: \C\times \C &\longrightarrow & \cO\\
       (y,z)  & \longmapsto & \frac{1}{r^2}\begin{pmatrix}
           \lambda r^2+\overline{s}z & \overline{s}(2r^2\lambda+\overline{s}z)\\
         z & -\lambda r^2-\overline{s}z
          \end{pmatrix},\quad s=\tfrac{y-z}{2\lambda},\quad r=\sqrt{1+s\overline{s}}.
\end{eqnarray*}
The   vector field on the domain of the chart coming from Theorem  \ref{thm:main} is
\[(y',z')=(2i\Im \lambda y, -2\Re \lambda z). \]
We check by hand that the $(21)$ entry of its pullback by the $(Y,Z)$ variables agrees with the $(21)$
entry of the Toda vector field. In other words, that
\[\left(\frac{z}{r^2}\right)'=\frac{-2z\Re(\lambda r^2+\overline{s}z)}{r^4},\quad
\mathrm{where}\quad y'=2i\Im \lambda y,\quad z'=-2\Re \lambda z.\]
Clearing denominators and using  $2rr'=(s\overline{s})'=2\Re(\overline{s}s')$, we are led to check
\[-2(\Re \lambda) zr^2-z(s\overline{s})'=-2z\Re(\lambda r^2+\overline{s}z)\Longleftrightarrow
\Re(\overline{s}s')=\Re(\overline{s}z).\]
Upon computing $s'$ and clearing $2\lambda\overline{\lambda}$ from the denominator we obtain
\[ \Re(\overline{(y-z)}(i(\Im\lambda)y+(\Re \lambda)z)=\Re(\overline{(y-z)}\lambda z).\]
Expanding the l.h.s and eliminating imaginary summands
\begin{align*}&\Re(i(\Im\lambda)y\overline{y}+\overline{y}(\Re \lambda)z-\overline{z}i(\Im \lambda)y-(\Re \lambda)z\overline{z})=\\ =&
\Re(\overline{y} \lambda z-(i(\Im \lambda) \overline{y}z-\overline{z}i(\Im \lambda)y)-(\Re \lambda)z\overline{z})=\\
=&\Re(\overline{y} \lambda z-(\Re \lambda)z\overline{z})=\Re(\overline{y} \lambda z)-(\Re \lambda)z\overline{z},
\end{align*}
we obtain the r.h.s., thus veryfing one (partial) instance of Theorem \ref{thm:main}.
\end{example}

\medskip

We finish this section by reviewing for $\sl_\C(2)$ the interpretation of Theorem \ref{thm:main} in the Introduction,  describing the Toda vector field as a compatible juxtaposition of certain diagonal linear  vector fields.

By the previous example, to the $(Y,Z)$ variables of the conjugacy class around $\Lambda\in \sl_\C(2)$ there corresponds  the  vector field
\[(y',z')=(2i\Im \lambda y, -2\Re \lambda z). \]
The first factor is a multiple of the rotational vector field, whereas the second one is a multiple of the Euler vector field. Note that we can prescribe these two factors and they  correspond to a unique diagonal matrix $\Lambda$. The other diagonal matrix in the conjugacy class is $-\Lambda$, and its corresponding  vector field is the opposite of the previous one.

Hence, the fixed  4-dimensional manifold $\cO$ has the following property: Given any multiples of the rotational and Euler vector fields, we can embed\footnote{The manifold $\cO$ has a canonical structure of holomorphic affine bundle over $\mathbb{CP}^1$ with fiber the affine complex line $\C$. However, as Example 1 shows, the embedding is not holomorphic.} $\C\times \C$ as an open dense subset of  $\cO$ in such a way that the given linear vector field on $\C\times \C$ extends to a vector field on $\cO$.

\section{Complex semisimple Lie groups}\label{sec:cx-ss}
Here we  prove Theorem \ref{thm:main}. This amounts to  generalizing  the results  in Section \ref{sec:cx} for the Toda vector field on $\sl_\C$ to an arbitrary complex semisimple Lie algebra $\gg$.

\medskip
We start by introducing additional structure on $\gg$ to define the Toda vector field.
Fix a complex semisimple Lie group\footnote{Our convention is that semisimple Lie groups are always connected.} $\G$ integrating $\gg$
and the following data:
\begin{enumerate}[(a)]
\item An anti-holomorphic Cartan involution $\theta:\G\to \G$
with fixed-point set the maximal compact subgroup $\K$.
The induced Lie algebra involution, also denoted by $\theta$, gives rise to the direct sum decomposition into $+1$ and $-1$ eigenspaces
 \[\gg=\kk\oplus \pp.\]
 \item A maximal abelian subalgebra $\aa$ of $\pp$ and a root ordering of the roots of $(\gg,\aa)$.
\end{enumerate}
The fixed data
 yields an Iwasawa factorization
 at the Lie algebra and Lie group levels \cite[Chapter VI, Section 4]{Kn}
 \[\gg=\kk\oplus \aa\oplus \nn,\quad \G=\K\A\NN,\]
 where $\nn$ is the sum of the positive root spaces of $(\gg,\aa)$ and $\A$,$\NN$ are the connected integrations of $\aa,\nn$, respectively.
In Section \ref{sec:cx}   we have used the classical Iwasawa factorization for $\G=\Sl_\C$ coming from the Gram-Schmidt algorithm  associated with the standard Hermitian inner product.

To an Iwasawa factorization on the complex semisimple Lie group $\G=\K\A\NN$
one
associates the Toda vector field \eqref{eq:Toda}
\[X'=[X,\pi_\kk X],\quad \mathrm{I}=\pi_\kk+\pi_\uu,\]
where $\uu$ is the Lie algebra of the  subgroup $\U=\A\NN$.

Because the Toda vector field has a Lax pair presentation,
it is tangent to all the orbits for the adjoint action of $\G$ on $\gg$. We are interested in those adjoint orbits $\cO$ that are semisimple and regular.
An element $X\in \gg$ is  semisimple if $[X,\cdot]:\gg\to \gg$ is a diagonalizable endomorphism,  or, equivalently,  if its adjoint orbit has nontrivial
 intersection with the Cartan subalgebra $\hh=i\aa\oplus \aa$; a matrix in $\sl_\C$ is a semisimple element if and only if it is diagonalizable. The regularity assumption means that the orbit has maximal
dimension among adjoint orbits of $\gg$. Regularity can be also defined in terms of the characteristic polynomial of the endomorphism $[X,\cdot]$; for matrices in $\sl_\C$ it amounts to having equal characteristic and minimal polynomial. The  analogs of diagonal matrices with simple spectrum are
elements  $\Lambda\in \cO\cap \hh$.

The other algebras and groups which are relevant for the proof of Theorem \ref{thm:main} are the generalization of the strictly lower triangular matrices, diagonal matrices, and the group of permutations.
\begin{itemize}
 \item the sum of the negative root spaces of $(\gg,\aa)$ is denoted by $\ll$ and its connected integration by $\L$. One also has $\L=\theta(\NN)$;
 \item the Cartan subgroup integrating $\hh$ is denoted by $\H$. It has a factorization $\H=\TT\A$
 into toroidal and vector parts that corresponds to the direct sum decomposition $\hh=i\aa\oplus \aa$;
 \item the Weyl group of $(\G,\aa)$, defined as the quotient of the normalizer of $\aa$ in $\K$ by the centralizer of $\aa$ in $\K$:
 \[\W=\NN_{\K}(\aa)/\Z_\K(\aa).\]
\end{itemize}

\medskip
We are ready for the proof of Theorem \ref{thm:main}.

\begin{proof}
 We construct $(Y,Z)$ variables for the Toda vector field on $\cO$ around $\Lambda\in \cO\cap \hh$. The proof follows  the one for $\sl_\C$ with the minor adjustments below.

\begin{enumerate}
 \item The subset $\cC\subset \K$ is now the {  \bf Chevalley big cell} associated with the fixed Iwasawa factorization $\G=\K\A\NN$, defined by  $\cC=\kappa(\L)$, where $\kappa$ is the first Iwasawa projection. It has the same factorization properties and geometric properties as in the case
 of $\Sl_\C$ (In \cite[Section 2]{MT} it is only asserted that $\cC\subset \K$ is an immersed submanifold, but the embedding property follows from the
 factorization $\cC\TT$ of an open dense subset of $\K$).
 \item The subset $\mathcal{D}\subset \U$ is defined by  \eqref{eq:LU-maps},  as for $\Sl_\C$. Lemma \ref{lem:more-LU} describing
 its geometric properties
 holds with  the same proof.
 \item With the submanifolds $\cC$ and $\mathcal{D}$ (and $\L$) in place, Lemma \ref{lem:LU-maps} holds with the same statement and proof. The key properties of $\Sl_\C$ used in that proof, that remain valid in the general setting of a complex semisimple Lie group $\G$, are the following:
 \begin{itemize}
  \item  $\H$ normalizes $\NN$ and $\L$;
  \item $\cC\TT$ is dense in $\K$;
  \item $\mathrm{Z}(\Lambda)=\H$;
   \item the adjoint action of a unipotent element of $\G$ does not fix any regular semisimple element of $\gg$;
  \item $\Lambda+\ll$ is preserved by the adjoint action of $\L$ and it consists of regular semisimple elements of $\gg$.
 \end{itemize}
 \item The definition of the $(Y,Z)$ variables
 \[\varphi:\cU\to (\Lambda+\ll)\times (\Lambda+\ll),\quad X\mapsto (\Lambda+Y,\Lambda+Z)\]
 is  the same as in \eqref{eq:local-coordinates-LU}, as is the proof of Corollary \ref{cor:local-coordinates-LU} showing that it is a real analytic diffeomorphism.
 \item Symes' factorization \eqref{eq:Toda-solution-Iwasawa-sl} of the solutions of the Toda flow holds for semisimple Lie groups with fixed Iwasawa factorization.
\item The  factorization in Theorem \ref{pro:Toda-solution-complex} for a solution starting at
$X\in \cU$ \[X=(\Lambda+Y)^{k^{-1}}=(\Lambda+Z)^u,\quad  Y,Z\in \ll,\, k\in \cC,\,u\in \mathcal{D}\]
 becomes
 \begin{equation*}
 X(t)
 ={\left((\Lambda+Y)^{\exp(-t\pi_{i\aa}\Lambda)}\right)}^{\kappa(l_{i\aa}(t))^{-1}},\quad l_{i\aa}(t)\in \L
 \end{equation*}
  \begin{equation*}
 X(t)
 ={\left((\Lambda+Z)^{\exp(t\pi_\aa\Lambda)}\right)}^{\nu(l_{\aa}(t))},\quad l_{\aa}(t)\in \L,
 \end{equation*}
where $\mathrm{I}|_\hh=\pi_{i\aa}+\pi_\aa$ is the analog of the decomposition of diagonal complex matrices into imaginary and
real parts. The proof of Theorem \ref{pro:Toda-solution-complex} uses exclusively that
$\H$ normalizes $\L$ and $\NN$. In particular, the  result relating the exponential of the sum to the product of the exponentials is valid
in an arbitrary Lie group $\G$ because of Zassenhaus' formula \cite{Ma}.
\item The linearizing properties of the $(Y,Z)$ variables are the same: The evolution decouples to the individual complex
1-dimensional root spaces $Y_\alpha$, $Z_\beta$, and it is given by multiples of the rotational and Euler vector fields, respectively (the multiple being
determined by the $\Lambda$ and by the negative root in question).
\item Finally, to prove that $\cU\cap \hh=\Lambda$ and that the collection of the $(Y,Z)$ variables cover $\cO$ by open, dense subsets, one needs to replace the permutation group by the Weyl group
of $(\G,\aa)$, and use standard properties of the Bruhat  open cover of $\G=\K\A\NN$ (see \cite[Section 3]{DKV} or \cite[Section 3]{MT}).
\end{enumerate}
\end{proof}

\subsection{Classical complex simple Lie groups}
Besides $\Sl_\C$, the other classical complex simple Lie groups are
 $\mathrm{SO}_\C$ and $\mathrm{Sp}_\C$, respectively the isotropy groups of
 \[z_1^2+\cdots +z_n^2,\qquad dz_1\wedge dz_{n+1}+\cdots +dz_n\wedge dz_{2n},\]
 for the action of $\Sl_\C$ on symmetric and skew-symmetric bilinear  forms.
Their Lie algebras are
\[\mathfrak{so}_\C=\{X\in \sl_\C\,|\, X=-X^T\},\quad \mathfrak{sp}_\C=\left\{\begin{pmatrix} X_1 & X_2\\ X_3 & -X_1^T\end{pmatrix}\,|\,X_2=X_2^T,X_3=X_3^T\right\}.\]

These subalgebras are also related to a natural constraint on the spectrum of a matrix: That it be invariant under multiplication times -1. Moreover, a skew-symmetric matrix (resp.  a symplectic matrix) is a regular semisimple element of $\mathfrak{so}_\C$ (resp. $\mathfrak{sp}_\C$) if an only if it is diagonalizable and has simple spectrum. Therefore, the relevant conjugacy classes in these subalgebras are obtained as intersections with conjugacy classes\footnote{The intersection is connected, i.e., it constitutes a unique $\So_\C$ (resp. $\Sp_\C$) conjugacy class.} of matrices in $\sl_\C$ with simple spectrum.

\medskip

Theorem \ref{thm:main} for these groups cannot be deduced straight away from the result for $\Sl_\C$ via the inclusion $\mathrm{SO}_\C,\mathrm{Sp}_\C\subset \Sl_\C$. A short calculation shows that
the Toda vector field on $\sl_\C$ is not tangent to the orthogonal and symplectic subalgebras.
However, they have conjugated subalgebras that are tangent to the Toda vector field. For instance, if rather than the above 2-form we consider
\[dz_1\wedge dz_{2n}+dz_2\wedge dz_{n-1}\cdots +dz_n\wedge dz_{n+1},\]
its isotropy Lie algebra $\mathfrak{q}$ integrates into a subgroup $\mathrm{Q}\subset \Sl_\C$ compatible with the Gram-Schmidt factorization in the following sense: The Gram-Schmidt factorization of a matrix in $\mathrm{Q}$ has the three factors in $\mathrm{Q}$ (see \cite[Table 1]{Sa} for the appropriate conjugation of the orthogonal groups). This implies that the Toda vector field is tangent to $\mathfrak{q}$ (see Lemma \ref{lem:compatible-Toda}). One can now deduce Theorem \ref{thm:main} for $\mathrm{Q}$ from Theorem \ref{thm:slC}.
Namely, that if $\Lambda\in \mathfrak{q}$ is a diagonal matrix with simple spectrum, then $\varphi((\Lambda+\ll\cap\mathfrak{q})\times (\Lambda+\ll\cap\mathfrak{q}))$ maps to an open dense subset of the $\mathrm{Q}$-conjugacy class of $\Lambda$,  and  the diagonal linear  vector field \eqref{eq:Toda-coordinates} restricts to another one on $(\Lambda+\ll\cap\mathfrak{q})\times (\Lambda+\ll\cap\mathfrak{q})$. Proving these statements uses ideas analogous to those in the proof of Theorem \ref{thm:main} for an arbitrary complex semisimple  Lie group, and has similar difficulty.

\medskip

These circle of ideas to obtain  linearization for orbits in subalgebras are pursued in detail  for real forms in Section \ref{sec:real forms}.

\begin{example}\label{ex:so} We  illustrate some aspects of Theorem \ref{thm:main} for the Toda vector field associated with a standard Iwasawa factorization for $\So_\C(5)$.

We first describe the choice of Iwasawa factorization. The Cartan involution $g\mapsto {(g^*)}^{-1}$
preserves $\So_\C$. This means that it restricts to a Cartan involution in $\So_\C$
with maximal compact subgroup
$\K\cap \mathrm{SO}_\C=\So_\R$.
Consider the following maximal abelian subalgebra of the skew-symmetric and Hermitian matrices $\so_\C\cap \pp$, and  choice of positive roots \cite[p. 127-128]{Kn}
\[\aa_0=\left\{\begin{pmatrix} 0 & ia_1 &  &  &  \\ -ia_1 & 0 & & & \\ & & 0 & ia_2 & \\
            & & -ia_2 & 0 & \\
            & & & & 0
           \end{pmatrix}\,|\, a_i\in \R
           \right\},\quad e_1,e_2,e_1-e_2,e_1+e_2,\]
 where $e_j(A)=a_j$. The corresponding root spaces are
 \[Y_\alpha=\begin{pmatrix} 0 &  Y_\alpha  &  \\ -Y_\alpha^T & 0 &  \\ & & 0 \end{pmatrix},\quad
        Y_{e_1-e_2}=\begin{pmatrix} 1 & i\\ -i & 1\end{pmatrix},\, Y_{e_1+e_2}=\begin{pmatrix} 1 & -i\\ -i & -1\end{pmatrix},\]
      \[Y_\alpha=\begin{pmatrix} 0 &  0 &  Y_\alpha^1\\ 0 & 0 & Y_\alpha^2 \\-{Y_\alpha^1}^T &-{Y_\alpha^2}^T & 0  \end{pmatrix},\quad
        Y_{e_1}^1=\begin{pmatrix} 1 \\ -i\end{pmatrix},\, Y_{e_1}^2=0,\,Y_{e_2}^1=0,\,Y_{e_2}^2=\begin{pmatrix} 1 \\ -i \end{pmatrix}.\]
The projection $\pi_{\so_\R(5)}$ associated with this Iwasawa decomposition sends $X\in \so_\C(5)$ to the real skew-symmetric matrix
\[\begin{split} &X-\Im\begin{pmatrix} 0 & x_{12} & 0 & 0 & 0 \\ -x_{21} & 0 & 0 & 0 & 0\\ 0& 0& 0& x_{34} & 0\\
           0 & 0 &  -x_{43} & 0&  0\\
           0 & 0 & 0 & 0 & 0
           \end{pmatrix}+\Im (x_{25}-ix_{15})\left(Y_{e_1}-Y_{e_1}^T\right)+\\&+
           \Im(x_{45}-ix_{35})\left(Y_{e_2}-Y_{e_2}^T\right)+
           \frac{1}{2}\Im\left((x_{23}-x_{14})+i(x_{13}+x_{25})\right)\left(Y_{e_1-e_2}-Y_{e_1-e_2}^T\right)+\\
           &+\frac{1}{2}\Im\left((x_{23}+x_{14})+i(x_{13}-x_{25})\right)\left(Y_{e_1+e_2}-Y_{e_1+e_2}^T\right).
           \end{split}\]
The (quadratic) Toda vector field
\begin{equation}\label{eq:Toda-so}X'=[X,\pi_{\so_\R(5)}X]
 \end{equation}
is tangent to the complex orthogonal conjugacy classes in $\so_\C(5)$. A regular semisimple orbit consists  of orthogonal matrices with fixed simple spectrum  $\{\pm \lambda_1,\pm \lambda_2\}$. It equals the conjugacy class of
\[\Lambda=\begin{pmatrix} 0 & i\lambda_1 &  &  &  \\ -i\lambda_1 & 0 & & & \\ & & & i\lambda_2 & \\
            & & -i\lambda_2 & 0 & \\
            & & & & 0
           \end{pmatrix}.\]
Each such conjugacy class has complex dimension 16 $= \mathrm{dim}\ \So_\C(5)-2$. It can be canonically identified with the cotangent bundle of the manifold of isotropic flags in $\C^5$ with respect to the fixed inner product. The latter compact manifold is a  fibration  over the manifold of isotropic planes in $\C^5$ with fiber $\mathbb{CP}^1$.

The conjugacy class of $\Lambda$ intersects the fixed Cartan subalgebra in four matrices in correspondence with the permutations that preserve the Cartan subalgebra. Theorem \ref{thm:main} provides $(Y,Z)$ variables centered at
the matrices
\[\Lambda, \Lambda^{(2,1,4,3,5)},\Lambda^{(3,4,1,2,5)},\Lambda^{(4,3,2,1,5)}\]
that transform \eqref{eq:Toda-so} in a diagonal linear vector field defined in two copies of the root spaces for the negative roots $-e_1,-e_2,-e_1-e_2,-e_1+e_2$. More explicitly, in $(Y,Z)$ variables centered at $\Lambda$ it corresponds to the  vector field whose eight complex variables evolve as
\[Y'_{-e_j}=[Y_{-e_j},\pi_{i\aa}\Lambda]=e_j(\Re\Lambda)Y_{e_j}=-i\Re\lambda_jY_{e_j},\]\[  Y'_{-e_1\mp e_2}=
-i(\Re(\lambda_1\pm\lambda_2))Y_{-e_1\mp e_2},
\]
\[Z'_{-e_j}=[Z_{-e_j},-\pi_{\aa}\Lambda]=-e_j(i\Im\Lambda)Z_{e_j}=-\Im\lambda_jZ_{e_j},\] \[ Z'_{-e_1\mp e_2}=
-(\Im(\lambda_1\pm\lambda_2))Z_{-e_1\mp e_2} \ .
\]
\end{example}

\section{Real forms and the Toda vector field}\label{sec:real forms}
In this section we  prove Theorems \ref{thm:relative} and \ref{thm:slC}, and Corollary \ref{cor:simple}. 

\medskip

We continue with the standing assumptions (a), (b) ---stated in Section \ref{sec:cx-ss}--- that provide
the Iwasawa factorization $\G=\K\A\NN$. A {\bf real form} $\G_0\subset \G$ is the fixed point set
of an anti-holomorphic group involution $\tau:\G\to \G$. For instance, $\K$ is a (compact) real form of $\G$.
We are interested
in real forms $\G_0$ with an Iwasawa factorization induced from that of $\G$. To that end
we further assume that
\begin{enumerate}[(a)]
\item[(c)] $\tau$ and the Cartan involution $\theta$ commute;
\item[(d)] the Cartan subalgebra $\hh=i\aa\oplus \aa\subset \gg$ is $\tau$-invariant and
the fixed-point set $\hh_0$ is a maximally non-compact Cartan subalgebra of $\gg_0$ (this means that the fixed-point set  $\aa_0$ for the action of $\tau$ on $\aa$ is a maximally non-compact abelian
subalgebra  of the fixed-point set $\pp_0$ for the action of $\tau$ on $\pp$ \cite[p. 386]{Kn});
\item[(e)] if $\alpha,\beta$ are roots of $(\gg,\aa)$ that agree on $\aa_0$, then $\alpha>0$  if and only if
$\beta>0$.
\end{enumerate}

Assumption (c) implies that $\theta$ restricts to $\G_0$
yielding the Cartan decomposition
\[\gg_0=\kk_0\oplus \pp_0,\quad \kk_0=\kk\cap \gg_0,\quad \pp_0=\pp\cap \gg_0.\]
Assumption (d) provides a maximal abelian subalgebra $\aa_0$ of $\pp_0$. One refers to the roots of
 $(\gg_0,\aa_0)$ as the restricted roots because any root of $(\gg_0,\aa_0)$
 is the restriction of a root of $(\gg,\aa)$ (or of several roots). Assumption (e)
 provides a root ordering for the restricted roots: A restricted root is positive if and only if it is the restriction
 of a positive root of $(\gg,\aa)$.
The outcome is the Iwasawa decomposition
\begin{equation*}\gg_0=\kk_0\oplus \aa_0\oplus \nn_0,\quad \G_0=\K_0\A_0\NN_0,
 \end{equation*}
where $\nn_0$ is the sum of positive root spaces for $(\gg_0,\aa_0)$ and $\A_0,\NN_0$ are the connected integrations of $\aa_0,\nn_0$, respectively.
We  refer to the data (a)-(e) as the data of {\bf compatible Iwasawa factorizations}
for $\G$ and  $\G_0$.

The Iwasawa decomposition of $\gg_0$ above defines the  Toda vector field on $\gg_0$.

\begin{lemma}\label{lem:compatible-Toda} Let $\G$ and $\G_0$ be a complex semisimple Lie group and a real form with
 compatible Iwasawa factorizations. Then the restriction to $\gg_0$ of the Toda vector field on $\gg$  equals the Toda vector field defined by the Iwasawa factorization of $\G_0$.
\end{lemma}
\begin{proof}
Because the Iwasawa factorizations are compatible
\[\kk_0\subset \kk,\quad \uu_0=\aa_0\oplus \nn_0\subset \uu,\]
we have
\[\pi_\kk|_{\gg_0}=\pi_{\kk_0}\quad (\mathrm{I}|_{\gg_0}=\pi_{\kk_0}+\pi_{\uu_0}),\]
and that implies the result.
\end{proof}

\begin{remark}If we are given the data (a), (b) that produces an Iwasawa factorization of $\G$, then Lemma \ref{lem:compatible-real-forms} in the Appendix shows that on any conjugacy
class of real forms there exists a representative $\G_0$ such that $\G$ and $\G_0$ have
compatible Iwasawa factorizations (i.e., satisfying (c)-(e)).
\end{remark}

\begin{remark}\label{rem:split}
Suppose that an anti-holomorphic involution $\tau$ on $\gg$ satifies (c) and (d) and, furthermore, that $\hh_0=\aa_0$, i.e., $\gg_0$ is a split real form of $\gg$. Then all roots of $(\gg,\aa)$ are real, and, therefore, (e) holds automatically  (see \cite[p. 390]{Kn} for the description of root types). In other words, a split real form $\G_0$ gets a compatible Iwasawa factorization.
 \end{remark}

The centralizer of $\aa_0$ in $\K_0$
\[\M_0=\mathrm{Z}_{\K_0}(\aa_0)=\{k\in \K\,|\,k X k^{-1}=X,\, \forall X\in \aa_0\}\]
plays a prominent role in the analysis of the Toda vector field on $\gg_0$.
\begin{lemma}\label{lem:automorphisms-real} Let $\G_0$ be a real form of a complex semisimple Lie group
with a given
Iwasawa factorization $\G_0=\K_0\A_0\NN_0$.
Then the Toda vector field on $\gg_0$ is invariant by conjugation by elements of $\M_0$.
\end{lemma}

\begin{proof}  A real form of a complex semisimple Lie group is a reductive real Lie group (with finite center). Therefore we can apply \cite[p. 456]{Kn} that asserts that $\M_0$ normalizes
$\uu_0$. Because $\M_0\subset \K_0$, it also normalizes $\kk_0$. Hence both Iwasawa Lie algebra projections
\begin{equation*}
 \pi_{\kk_0}:\gg_0\to \kk_0,\quad \pi_{\uu_0}:\gg_0\to \uu_0
\end{equation*}
are equivariant by conjugation by elements of $\M_0$. This implies that the Toda vector field on $\gg_0$ is preserved
by conjugation by elements of $\M_0$.
\end{proof}

\begin{remark}\label{rem:automorphisms} For a complex semisimple Lie group $\G=\K\A\NN$ the centralizer of $\aa$ in $\K$ agrees with the toroidal part $\TT$ of the Cartan subgrup $\H$. The proof of Lemma \ref{lem:automorphisms-real}
 also shows that conjugation by elements of $\TT$ produces automorphisms of the Toda
 vector field on $\gg$ (this is implicit in the proof of Theorem \ref{thm:linear-coordinates}).
\end{remark}

\subsection{Induced $(Y,Z)$ variables on maximally non-compact adjoint orbits}

Let $\G_0=\K_0\A_0\NN_0$ be a real form of a complex semisimple Lie group $\G=\K\A\NN$ with compatible Iwasawa factorization. We  introduce the adjoint $\G_0$ orbits in $\gg_0$ that concern us, and  discuss their relation with adjoint $\G$ orbits.

\medskip

An element of $\gg_0$ is regular semisimple in $\gg_0$ if and only if it is regular semisimple in $\gg$ \cite[p. 406]{Ro}. Therefore a regular semisimple adjoint $\G_0$ orbit in $\cO_0\subset \gg_0$ is a   connected component of the intersection of a regular semisimple adjoint orbit $\cO\subset \gg$ with $\gg_0$.
Recall that by assumptions (c) and (d) the Cartan subalgebra $\hh\subset \gg$ is invariant by $\tau$ and its fixed-point set
$\hh_0$ is a maximally non-compact Cartan subalgebra of $\gg_0$.
A regular semisimple adjoint $\G_0$ orbit $\cO_0\subset \gg_0$ is {\bf maximally non-compact}
if its elements belong to maximally non-compact Cartan subalgebras. Because all maximally non-compact Cartan subalgebras of $\G_0$ are conjugated \cite[Proposition 1.2]{Ro}, this is equivalent to saying that $\cO_0$ is a regular adjoint $\G_0$ orbit
such that $\cO_0\cap \hh_0\neq \emptyset$. Equivalently, $\cO_0$ is a connected component of $\cO\cap \gg_0$ that intersects $\hh_0$ non trivially, for any regular adjoint $\G$ orbit $\cO$.

\medskip

We now adapt the construction of the rulings and of its $(Y,Z)$ variables described in Section \ref{sec:factorizations}. The projections associated with the Iwasawa factorization $\G_0=\K_0\A_0\NN_0$ give rise to the Chevalley
big cell $\cC_0=\kappa(\L_0)\subset \K_0$, where $\L_0=\theta(\NN_0)$ with Lie algebra denoted by $\ll_0$, and to the subset $\mathcal{D}_0\subset \U_0$, defined as in \eqref{eq:more-LU}. However, the cell $\cC_0$ does not provide in general enough $Z$ variables. The cell has to be enlarged by means of the Chevalley big cell $\cC_{\M_0}\subset \M_0$
obtained from the compatible Iwasawa factorizations, as described in  Lemma \ref{lem:Iwasawa-compatible} the Appendix.

\begin{lemma}\label{lem:relative-LU-maps} Let $\Lambda\in \cO_0\cap \hh_0$. The maps given by
\[\Phi_{\cC_0}: \cC_{\M_0}\times \cC_0\times \L_0\to \cO_0,\quad (m,k,l)\mapsto \Lambda^{(lkm)^{-1}},\]
\[\Phi_{\mathcal{D}_0}:\cC_{\M_0}\times \mathcal{D}_0\times \L_0\to \cO_0,\quad
(m,u,l)\mapsto \Lambda^{m^{-1}ul^{-1}}
\]
have the following properties.
\begin{enumerate}[(i)]
 \item They are real analytic diffeomorphism onto an open dense subset $\cU_0\subset \cO_0$; for any $m\in\cC_{\M_0}$ their
 restriction to $\{m\}\times  \cC_0\times \L_0$ and $\{m\}\times  \mathcal{D}_0\times \L_0$
 have equal image $\cU_{0,m}$.
 \item They produce rulings of $\cU_0$ by
$\cC_{\M_0}\cC_0$ and $\cC_{\M_0}\mathcal{D}_0$ translates of
\[
\Lambda+\ll_0=\Phi_{\cC_0}(\{e\}\times\{e\}\times\L_0)=\Phi_{\mathcal{D}_0}(\{e\}\times\{e\}\times\L_0).
\]
\end{enumerate}
\end{lemma}
\begin{proof}
Let $\TT_0$ be the centralizer of $\tt_0$ in $\M_0$, $\TT_0=\mathrm{Z}_{\M_0}(\tt_0)$.
By item (ii) in Lemma \ref{lem:Iwasawa-compatible} , the subset
\[\cC_{\M_0}\cC_{0}\TT_0\L_0\A_0\subset \G_0\]
is open and dense and equals $\cC_{\M_0}\cC_0\L_0\mathrm{Z}_{\G_0}(\Lambda)$.
The rest of the proof of (i) is identical to that of Lemma \ref{lem:LU-maps}. To prove (ii)
we need to ensure that $\L_0$ acts freely on regular semisimple elements. The argument
follows Lemma \ref{lem:LU-maps}. It uses two additional facts.

\begin{itemize}
	\item 
 The centralizer in $\G_0$ of a regular semisimple element
has finitely many connected components.
\item The connected component containing the identity consists of semisimple elements. 
\end{itemize}
Thus the centralizer cannot contain nontrivial unipotent elements because the latter have infinite order.
\end{proof}

By Lemma \ref{lem:Iwasawa-compatible}, the Chevalley big cell $\cC_{\M_0}$ is contained in the Chevalley big cell $\cC$ of the
complex Lie group $\G$. Furthermore, given $m\in\cC_{\M_0}$,
\begin{itemize} \item it   has a $\L\U$ factorization
 \[m=\zeta\mu,\quad \zeta\in \L_\Im,\, \ \mu\in \M\cap \U.\] Here
 $\L_\Im\subset \L\subset \G$  integrates the negative imaginary root spaces $\ll_{\Im}$. \item The restriction of the diffeomorphism $\Phi_\cC$ \eqref{eq:LU-maps}
\[\Phi_{\cC}:\cC_{\M_0}\to \Lambda+\ll_\Im,\quad m\mapsto \Lambda^{\zeta^{-1}}\]
is a diffeomorphism onto its image. 
\end{itemize}
We are ready to define the map $\varphi_0$ yielding the induced $(Y,Z)$ variables.
\[ \varphi_0 = (\Phi_{\cC_0}\circ \mathrm{pr}_3\circ \Phi_{\cC_0}^{-1},\Phi_{\mathcal{D}_0}\circ \mathrm{pr}_3\circ \Phi_{\mathcal{D}_0}^{-1}+\Phi_\cC\circ \mathrm{pr}_1\circ \Phi_{\mathcal{D}_0}^{-1}-\Lambda):\]
\vspace{-0.5cm}
\begin{eqnarray*}
& \cU_0 &\to  (\Lambda+\ll_0)\times (\Lambda+\ll_0+\ll_\Im) \nonumber \\
& X&\mapsto (Y,Z+Z_\Im),
\end{eqnarray*}
where 
\[ X=\Lambda^{(l_1km)^{-1}}=\Lambda^{m^{-1}ul_2^{-1}},\quad Y=\Lambda^{l_1^{-1}}, \quad Z=\Lambda^{l_2^{-1}}, \quad Z_\Im=\Lambda^{\zeta^{-1}}-\Lambda \ . \]
The  big cell $\cC_{\M_0}$ provides the additional induced $Z_\Im$ variables. As $\M_0$ acts by automorphisms on the Toda flow, these variables evolve trivially, as shown below.

\begin{theorem}\label{thm:real-form-local-coordinates-LU} Let $\cO_0\subset \gg_0$ be a maximally
 non-compact regular semisimple adjoint orbit and let $\Lambda\in \cO_0\cap \hh_0$.
 The map
 \[\varphi_0:\cU_0\to (\Lambda+\ll_0)\times (\Lambda+\ll_0+\ll_\Im)\]  has the following properties.
\begin{enumerate}[(i)]
 \item It is a diffeomorphism.
 \item If $\cO$ is the adjoint $\G$ orbit through $\Lambda$ and $\varphi$ provides the $(Y,Z)$ variables around $\Lambda\in \cO$ from \eqref{eq:local-coordinates-LU}, then, in the notation of Lemma \ref{lem:relative-LU-maps}, 
 \[\mathrm{pr}_1\circ\varphi|_{\cU_0}=\mathrm{pr}_1\circ \varphi_0,\quad
\varphi|_{\cU_{0,e}}=\varphi_0|_{\cU_{0,e}}.\]
 \item It takes the Toda vector field on $\cU_0$ to the linear vector field
 \begin{equation*}
(Y',Z',Z_\Im')=\left([Y,\pi_{\tt_0}\Lambda],[Z,-\pi_{\aa_0}\Lambda],0\right),\quad \mathrm{I}|_{\hh_0}=\pi_{\tt_0}+\pi_{\aa_0}.
\end{equation*}
\end{enumerate}
\end{theorem}
\begin{proof}
  To exhibit  $\varphi_0^{-1}$ we first apply $\Phi_\cC^{-1}$ to $\Lambda+Z_\Im\in \Lambda+\ll_\Im$ to write
  \[
   \Lambda+Z_\Im=\Lambda^{\zeta^{-1}}, \quad \zeta=m\mu^{-1}.
  \]
  The remaining steps in the construction of the inverse are as in Corollary \ref{cor:local-coordinates-LU}.

  By Lemma \ref{lem:Iwasawa-compatible}
 the compatibility of Iwasawa factorizations  yields the inclusions
  \[\cC_0\cC_{\M_0}\subset \cC,\quad \mathcal{D}_0\subset \mathcal{D}.\]
  The first one implies  the first compatibility condition
  between $\varphi$ and $\varphi_0$ in item (ii); combined with the second inclusion,
  it provides the second compatibility condition.

  To prove item (iii), we  claim that the Toda vector field on $\cO_0$ is tangent to $\cU_{0,e}$.
  Assuming this claim, because by Lemma \ref{lem:automorphisms-real} the adjoint action of $\M_0$ is by  automorphism of the Toda vector field,
 the Toda vector field is tangent to the submanifolds $\cU_{0,m}$, for all $m\in \M_0$.
 By the construction of induced $(Y,Z)$ variables,
 \[\varphi_0(\cU_{0,m})=(\Lambda+\ll_0)\times (\Lambda+\ll_0)+Z_{\Im}(m),\]
 which immediately implies that $Z'_{\Im}=0$.
 Since for any $X\in \cU_{0,e}$ its $Y$ and $Z$ coordinates agree with those of $X^{m^{-1}}$,
 the evolution of $Y,Z$ does not depend on the values of $Z_{\Im}$, and this proves the formula for the Toda vector field in (iii). 

 To prove the claim we argue as follows.
\begin{itemize}
 \item Because the Iwasawa factorizations of $\G_0$ and $\G$ are compatible, by Lemma \ref{lem:compatible-Toda} the Toda
 vector field on $\cO_0$ agrees with the restriction of the Toda vector field on $\cO$, so we may assume that we work with the latter.
\item By item (ii),
 $\varphi_0$ and $\varphi$ agree on $\cU_{0,e}$. Thus, we can  use  $\varphi$.
 \item By Theorem \ref{thm:linear-coordinates}, $\varphi$ takes the Toda vector field on $\cU$   to the  vector
field
\[(Y',Z')=([Y,\pi_{i\aa}\Lambda],[Z,-\pi_{\aa}\Lambda])\]
on $(\Lambda+\ll)\times (\Lambda+\ll)$. 
As $\Lambda\in \hh_0$, we have $\pi_{i\aa}\Lambda\in \tt_0$,
$\pi_{\aa}\Lambda\in \aa_0$. Since both $\tt_0$ and $\aa_0$ normalize $\ll_0$ it follows that this vector field is tangent to $(\Lambda+\ll_0)\times (\Lambda+\ll_0)$,
  and this finishes the proof of the claim.
\end{itemize}

Note that the  vector field in (iii) is diagonal in the coordinate set $Z$ (and in $Z_\Im$, of course). The reason is that each restricted root space is a real form of the sum of roots spaces of $(\gg,\aa)$ for the roots that agree on $\aa_0$. Since $\pi_{\aa_0}\Lambda\in \aa_0$, the restriction of $Z'=[Z,-\pi_{\aa_0}\Lambda]$ to the each restricted root space is a multiple of the Euler vector field. 

However, it is not clear whether  $Y'=[Y,\pi_{\tt_0}\Lambda]$ may be diagonal in general.
\end{proof}

\begin{proof}[Proof of Theorem \ref{thm:relative}]
First, we show that the Toda vector field on $\gg$ is tangent to $\cO_0$. By Lemma \ref{lem:compatible-Toda}  it is tangent to $\gg_0$. Since it is also tangent
to the adjoint orbit $\cO$ through $\Lambda$, it is tangent to  $\cO\cap \gg_0$, and, in particular, to $\cO_0$, which is its connected
component  through $\Lambda$.

The existence of the induced $(Y,Z)$ variables (the local coordinates alluded to in item (ii) of  Theorem \ref{thm:relative}) is the content of Theorem \ref{thm:real-form-local-coordinates-LU}. The density of their domain is proved in item (i) in Lemma \ref{lem:relative-LU-maps}.
Also, the inclusion $\cC_{\M_0}\cC_0\subset \cC$ implies that $\cU_0\subset \cU$. Since  $\cU\cap \hh=\{\Lambda\}$ we deduce that $\Lambda$ is the only element in $\cU_0\cap \hh_0$.
So it
only remains to prove that the union of these charts centered at
points of $\cO_0\cap \hh_0$ cover $\cO_0$. To do that we combine the Bruhat covers of $\G_0$ and $\M_0$,
\begin{align*}
\G_0 &=\bigcup_{w\in \W_0}\M_0\NN_0\A_0\L_0 w=\bigcup_{w\in \W_0}
\left(\bigcup_{\sigma\in  \W_{\M_0}}\cC_{\M_0} \Z_{\M_0}(\tt_0)\sigma\right)\NN_0\A_0\L_0 w=\\
&=\bigcup_{w\in \W_0}\bigcup_{\sigma\in  \W_{\M_0}}\cC_{\M_0}\NN_0\L_0\A_0 \Z_{\M_0}(\tt_0)\sigma w,
\end{align*}
where in the last equality we have used that $\sigma,\Z_{\M_0}(\tt_0)\subset \M_0$ normalize $\NN_0\A_0\L_0$.
Once an element $\Lambda\in \cO_0\cap \hh_0$ is chosen, this cover of $\G_0$ produces, upon taking the adjoint action, the corresponding one of $\cO_0$.
We now verify that the subsets of the latter cover coincide with the domains of induced $(Y,Z)$ variables. We also provide yet another parametrization of the elements in $\cO_0\cap \hh_0$.

\begin{align*}
\cO_0 &=\bigcup_{w\in \W_0}\bigcup_{\sigma\in \W_{\M_0}}\Lambda^{\cC_{\M_0}\NN_0\L_0\A_0 \Z_{\M_0}(\tt_0)\sigma w}=
\bigcup_{w\in \W_0}\bigcup_{\sigma\in \W_{\M_0}}{(\Lambda^{\sigma w})}^{\cC_{\M_0}\NN_0\L_0}=\\
&=\bigcup_{w\in \W_0}\bigcup_{\sigma\in \W_{\M_0}}{(\Lambda^{\sigma w}+\ll_0)}^{\cC_{\M_0}\NN_0}
=\bigcup_{w\in \W_0}\bigcup_{\sigma\in \W_{\M_0}}{(\Lambda^{\sigma w}+\ll_0)}^{\cC_{\M_0}\cC_0}=\\
&=\bigcup_{\rho\in \NN_{\K_0}(\hh_0)/\Z_{\K_0}(\hh_0)}{(\Lambda^{\rho}+\ll_0)}^{\cC_{\M_0}\cC_0},
\end{align*}
where the last equality follows from
Lemma \ref{lem: Weyl short exact} in the Appendix.

The above equalities show that the cardinalities of $\cO_0\cap \hh_0$ and $\NN_{\K_0}(\hh_0)/\Z_{\K_0}(\hh_0)$ agree. This is consistent with \cite[Proposition 2.9]{Ro}.
\end{proof}

\begin{remark} We have presented a relative version of the  linearization result in which
 the real reductive Lie group with finite center $\G_0$ is a real form of a complex semisimple Lie group.
 Our methods can be adapted to show that, for any real semisimple Lie group with a fixed Iwasawa factorization, the ensuing
 Toda vector field linearizes on maximally non-compact regular semisimple orbits.
 Note that this is not the same result proved in \cite{MT},  which addressed the diagonalization of the Toda vector field on hyperbolic $\K_0$ orbits of maximal dimension (Lie theoretic generalizations of conjugacy classes of real symmetric matrices with simple spectrum). Recall that a semisimple element of $\gg_0$ is hyperbolic if all its eigenvalues for its adjoint action on $\gg_0$ are real. Equivalently, the element is conjugated to one in $\aa_0\subset \hh_0$. Note that $\aa_0$ contains regular elements if and only if  $\mm_0$
 is abelian (for instance,  if the real semisimple Lie algebra is split or is a complex Lie algebra). Therefore, the elements addressed in \cite{MT} are not regular in general (besides not working with their full adjoint orbit but with its intersection with $\pp_0$).
\end{remark}

\medskip

\begin{proof}[Proof of Theorem \ref{thm:slC}]
 The conjugacy classes of the non-compact real forms of $\Sl_\C$ have
the following classical representatives \cite[p. 451-452]{He}:
\[\Sl_\R,\quad  \Sl_\mathbb{H},\quad \mathrm{SU}(p,q).\]
In fact, \cite{He} works with the real form $\Su^*$ which is conjugated to $\Sl_\mathbb{H}$. The former is the result of  embeding a quaternionic $n\times n$ matrix as a complex matrix with four $n\times n$ blocks, whereas the latter embeds it as $n\times n$ matrix whose elements are $2\times 2$ complex matrices.

{\bf $\Sl_\R$:} The proof for the split real form of $\Sl_\C$ is the most straightforward one.
 The Gram-Schmidt factorization
of $\Sl_\C$ is clearly compatible with $\Sl_\R$. The compatibility can also be drawn from Remark \ref{rem:split}, and the fact that the involution $\tau(X)=\overline{X}$ that determines $\sl_\R$ commutes with the Cartan
involution $\theta(X)=-X^*$ and preserves the diagonal matrices of $\sl_\C$.

The $\Sl_\R$ conjugacy class of a diagonal real matrix with simple spectrum  equals the intersection of its $\Sl_\C$ conjugacy class with $\sl_\R$, and both orbits contain the same diagonal matrices. Because the centralizer in $\Sl_\R$ of the diagonal matrices is discrete, i.e., its Lie algebra $\mm$ is trivial, by Theorem \ref{thm:real-form-local-coordinates-LU}
the $(Y,Z)$ variables on $\cO$ centered at any diagonal $\Lambda\in \cO$ restrict to $(Y,Z)$ variables for the Toda vector field on $\cO_0$. Their image  is $(\Lambda+\ll_0)\times (\Lambda+\ll_0)$. Such variables, without reference to complex matrices, were introduced in \cite{LMST}.

The previous considerations for $\sl_\R$ hold for a split real form $\gg_0$ of a complex semisimple Lie algebra,  defined by an involution that commutes with the fixed Cartan involution on $\gg$ and preserves the fixed Cartan subalgebra of $\gg$.

{\bf $\Sl_\mathbb{H}$:} The Lie algebra of this real form  is the fixed-point set of the involution
\[\tau(X)=-gJ\overline{X}Jg^{-1},\quad J=\begin{pmatrix} 0 & \mathrm{I}_n\\ -\mathrm{I}_n & 0\end{pmatrix},\]
where $g$ is the permutation matrix of $(                                                                                                      1, 3,  5, \cdots , 2n-1,  2, 4,  6,\cdots , 2n)$,
\[\sl_\mathbb{H}=\left\{\begin{pmatrix} q_{11} &   &q_{nn}\\ & \cdots &\\
                         q_{n1} & & q_{nn}
                        \end{pmatrix}\in \sl_\C,\quad q_{ab}=
                        \begin{pmatrix}x_{2a-1,2b-1} & x_{2a-1,2b} \\
                         -\overline{x_{2a-1,2b}} & \overline{x_{2a-1,2b-1}}
                        \end{pmatrix}
                        \right\}\]

We prove that $\Sl_\mathbb{H}$ obtains a compatible Iwasawa factorization from the Gram-Schmidt factorization, i.e., that properties (c)-(e) hold.
The involution $\tau$ commutes with the Cartan involution  and preserves  diagonal matrices in $\sl_\C$.
The diagonal matrices in $\sl_\mathbb{H}$ and its vector part are
\[
 \hh_0=\left\{\begin{pmatrix} \lambda_1 & 0 & &  &  \\ 0 &
 \overline {\lambda}_1 &  & & \\
 & & \cdots &  & \\
  & & & \lambda_n & 0 \\
  & &  & 0 & \overline{\lambda}_n
  \end{pmatrix}\right\},\, \quad \aa_0=\left\{\begin{pmatrix} a_1 & 0 & &  &  \\ 0 &
 a_1 &  & & \\
 & & \cdots &  & \\
  & & & a_n & 0 \\
  & &  & 0 & a_n
  \end{pmatrix}\right\},
\]
for $\lambda_i\in \C,\,a_i\in \R$. The Cartan subalgebra $\hh_0$ is maximally non-compact because  $\sl_\mathbb{H}$ has just one conjugacy class of Cartan subalgebras \cite[p. 398-399]{Su}. Finally, if two roots of $(\sl_\C,\hh)$ agree on $\aa_0$, then they  must be of the form
\[\alpha=\pm (e_i-e_j),\,\beta=\pm (e_{i+\varepsilon_i}-e_{j+\varepsilon_j}),\quad \varepsilon_k=\begin{cases} 1,& k\,\,\mathrm{odd}\\ -1, & k\,\,\mathrm{even}\end{cases}.\]
Since $\alpha>0$ if and only if $\beta>0$, the root ordering of $(\sl_\C,\aa)$ is compatible with $\sl_\mathbb{H}$. The sum of the negative (restricted) root spaces is
\[\ll_0=\left\{\begin{pmatrix} 0 & 0 & &   0 \\ q_{21} &
 0 &  & 0 \\
 &  & \cdots   &  \\
q_{n1 }& q_{n2} &   & 0
  \end{pmatrix}\in \sl_\C\right\}.\]

If $\Lambda\in \hh_0$ has simple spectrum, then its $\Sl_\mathbb{H}$ conjugacy class $\cO_0$ is the intersection of its $\Sl_\C$ conjugacy class $\cO$ with $\sl_\mathbb{H}$ (i.e. the latter intersection is connected \cite[Corollary 2.12]{Ro}). The diagonal matrices in $\cO_0$ are parametrized by those permutations of $2n$ elements such that $\sigma(j+1)=\sigma(j)+1$, if $j$ and $\sigma(j)$ are odd, and $\sigma(j+1)=\sigma(j)-1$, if $j$ is odd and $\sigma(j)$ is even. According to Lemma \ref{lem: Weyl short exact}, they can be arranged as an extension of permutations of $\Re \Lambda$,  the permutation group in $n$ elements, by the permutations fixing $\Re \Lambda$,  the group  $\mathbb{Z}_2^n$. The latter is the Weyl group  of the centralizer of $\Re\Lambda$ in $\Sl_\mathbb{H}\cap \Su_\C$. This centralizer equals $n$ diagonal copies of $\Su(2)$ (unit quaternions), and its  Lie algebra is
\[\mm_0=\left\{\begin{pmatrix} z_{11} & -\overline{z}_{21} & &  &  \\
z_{21}  & -z_{11} & & \\
 & & \cdots &  & \\
  & & & z_{2n-1,2n-1} & -\overline{z_{2n,2n-1}} \\
  & &  &z_{2n,2n-1}  & -z_{2n-1,2n-1}
  \end{pmatrix}\in \su\right\}.\]
The negative imaginary root spaces are the subdiagonal positions $z_{2a,2a-1}$.
By Theorem \ref{thm:real-form-local-coordinates-LU}
the restriction to $\cO_0$ of the $(Y,Z)$ variables on $\cO$ centered at any diagonal $\Lambda\in \cO_0$, after a suitable modification, yield induced $(Y,Z)$ variables for the Toda vector field on $\cO_0$. Explicitly, the  vector field that we obtain is
\[\begin{split}  Y'_{2a-1,2b-1} &= \Im\left(\lambda_{b}-\lambda_{a}\right)Y_{{2a-1,2b-1}},\quad Y'_{2a-1,2b}= \Im\left(\overline{\lambda_{b}}-\lambda_a\right)Y_{{2a-1,2b}},\quad a<b,\\
 Z'_{2a-1,2b-1} & = \Re\left(\lambda_{a}-\lambda_{b}\right)Z_{{2a-1,2b-1}},\quad Z'_{2a-1,2b}= \Re\left(\lambda_{a}-\overline{\lambda_{b}}\right)Z_{{2a-1,2b}},\quad a<b,\end{split}\]
 \begin{equation*}
 Z'_{2a,2a-1}=0.
 \end{equation*}
It evolves separately on 1-dimensional complex vector spaces (so it is diagonal). We do not know the reason why the evolution in the variable $Y$ is diagonal (for instance, whether this may be the case for  real forms with a unique conjugacy class of Cartan subalgebras).

{\bf $\mathrm{SU}(p,q)$ ($p\geq q$):} Its Lie algebra is the fixed-point set of 
\[\tau(X)=\mathrm{I}_{p,q}X^*\mathrm{I}_{p,q},
\quad \mathrm{I}_{p,q}=\begin{pmatrix}
 \mathrm{I}_p & 0 \\ 0 & -\mathrm{I}_q
\end{pmatrix}.\]
\[\mathfrak{su}(p,q)=\left\{\begin{pmatrix}
 X_p & V \\ V^* & X_q
\end{pmatrix}\,|\, X_p\in \mathfrak{u}(p),\,X_q\in \mathfrak{u}(q),\,\mathrm{tr}(X_p+X_q)=0\right\}.
\]
This involution commutes with $\theta$ but it does not preserve the diagonal complex matrices. A maximally non-compact Cartan subalgebra of $\mathfrak{su}(p,q)$ is
\[\mathfrak{d}_0=\left\{\begin{pmatrix} C_{p-q} & 0 & 0\\
    0 & B_{q} & A_q^{\mathrm{op}} \\
    0 & {A_q^\mathrm{op}}^T & B^\mathrm{op}_q
  \end{pmatrix}\,|\,
  \quad A^\mathrm{op}_q=\mathrm{I}^\mathrm{op}A_q\mathrm{I}^\mathrm{op}\right\}, \qquad \mathrm{I}^\mathrm{op}=\begin{pmatrix}
  	0 &\cdots & 1\\
  	& \cdots &\\
  	1 &\cdots & 0
  \end{pmatrix}
\]
where $C_{p-q},B_q$ range over purely imaginary diagonal matrizes and $A_q$ ranges over
real diagonal matrices (c.f. \cite[p. 371-372]{Kn}). Its conjugation by the matrix $c$ in the statement of Theorem \ref{thm:slC} 
takes $\mathfrak{d}_0$ to a real form $\hh_0$ of the diagonal matrices
 \[\hh_0=\left\{\begin{pmatrix} \Lambda_{q} & 0 & 0   \\ 0 & \Im_{p-q} & 0\\
 0 & 0 & -\overline{\Lambda}_q^\mathrm{op}
  \end{pmatrix}\right\},\]
where  $\Lambda_q$ and $\Im_{p-q}$ are square matrices
of complex and purely
  imaginary complex numbers, respectively.
In fact, the conjugation by the second factor
in $c$ already does this. However, we need the permutation matrix so that the ordering of the eigenvalues guarantees that the ordering of $(\sl_\C,\aa)$ induces an ordering of $(\gg_0=c\su(p,q)c^{-1},\hh_0)$.

If  $\Lambda\in \hh_0$ has simple spectrum, then the intersection of its $\G_0$ conjugacy class with $\hh_0$ is parametrized by those permutations of $p+q$ elements that preserve $\hh_0$.  They are an extension of permutations of $\Re \Lambda$,  itself  an  extension by $\mathbb{Z}_2$ of  the group permutations of $q$ elements, by  permutations fixing $\Re \Lambda$,  the group of permutations of  $p-q$ elements.
    Theorem \ref{thm:real-form-local-coordinates-LU}
gives induced $(Y,Z)$ variables for the Toda vector field on the $\G_0$ conjugacy class around  $\Lambda$, by restriction of the coordinates on  the $\Sl_\C$ conjugacy class and adjustment of the variable $Z$.
\end{proof}

\begin{proof}[Proof of Corollary \ref{cor:simple}]
We regard $\sl_\R,\sl_\mathbb{H},\sp_{\C},\so_\C$ as subalgebras of $\sl_\C$, where we abuse notation by taking $\sp_\C$ to be the isotropy Lie algebra of
\[dz_1\wedge dz_{2n}+dz_2\wedge dz_{n-1}\cdots +dz_n\wedge dz_{n+1}.\]
(The subalgebra denoted by $\mathfrak{q}$ in Example \ref{ex:so}).
For the first three subalgebras we consider the classical Toda vector field on $\sl_\C$; for $\so_\C$ we have to conjugate it so that it becomes tangent to the subalgebra (the conjugation is dictated by the choice of Iwasawa decomposition of $\so_\C$ described in Example \ref{ex:so}).

In all cases  the relevant conjugacy classes in the subalgebras are the (connected) intersection of the subalgebra with a conjugacy class of matrices in $\sl_\C$ with simple spectrum subject to additional contraints: For $\sl_\R$ the spectrum has to be real, for $\sl_\mathbb{H}$ it has to be invariant by conjugation and not include real numbers, and for $\sp_\C,\so_\C$ it has to be invariant by muliplication by -1. For $\sl_\mathbb{H}$ and $\sp_\C$ there is an additional ordering constraint for a diagonal matrix with the given spectrum to be in the subalgebra. For $\so_\C$ the matrices around which charts are constructed are not diagonal; they have diagonal $2\times 2$ blocks of the form
\[\begin{pmatrix} 0 & i\lambda \\ -i\lambda & 0 \end{pmatrix}\]
(And an additional 0 in the diagonal for the odd dimensional case).

In all cases the corresponding  vector fields are diagonal. This is a consequence of Theorem \ref{thm:main} for $\sl_\C,\sp_\C$ and $\so_\C$, and of the proof of Theorem \ref{thm:slC} for $\sl_\R$ (for split real forms) and for $\sl_\mathbb{H}$ (explicit computation).
\end{proof}

\section{Appendix: Compatible Iwasawa factorizations}\label{sec:appendix}
We  discuss the existence of compatible Iwasawa factorizations, and the relation between the different Chevalley big cells and Weyl groups that arise in the presence of such Iwasawa factorizations. These results enter in the proof of Theorem \ref{thm:relative}.

\medskip

Let $\G$ and $\G_0$ be a complex semisimple Lie group and a real form endowed with  compatible Iwasawa
factorizations. As in Sections 3 and 4, this  means that we have an Iwasawa factorization of $\G=\K\A\NN$ provided
by conditions (a) and (b) described at the beginning of Section \ref{sec:factorizations}: An anti-holomorphic Cartan involution
\[\theta:\G\to \G,\quad \gg=\kk\oplus \pp\]
and a root ordering of the roots of $(\gg,\aa)$, where $\aa$ is a maximal abelian subalgebra of $\pp$.
Furthermore, the anti-holomorphic group involution $\tau$ whose fixed-point set is $\G_0$ satisfied conditions (c)-(e)
described at the beginning of Section \ref{sec:real forms}: $\tau$ and $\theta$ commute, the Cartan subalgebra
$\hh=i\aa\oplus \aa$ is invariant by $\tau$, $\hh_0=\hh^\tau$ is a maximally non-compact Cartan subalgebra of $\gg_0$,
and the root ordering of $(\gg,\aa)$ induces a root ordering of the restricted roots of $(\gg_0,\aa_0)$,
where $\aa_0=\aa^\tau$. As a result $\theta$ acts on $\G_0$ and gives a Cartan decomposition
\[\gg_0=\kk_0\oplus \pp_0,\quad \kk_0=\kk\cap \gg_0,\,\pp_0=\pp\cap \gg_0,\]
and the choice of $\aa_0$ and the induced root ordering defines an Iwasawa factorization for $\G_0=\K_0\A_0\NN_0$.
Since
\begin{equation*}\label{eq:Iwasawa-compatible}
\K_0\subset\K,\quad \A_0\subset\A,\quad \NN_0\subset \NN,
\end{equation*}
 for any $g\in \G_0$ its Iwasawa factorization in $\G_0$ and $G$ give the same factors.

\medskip

The following lemma obtains real forms with compatible Iwasawa factorizations.

 \begin{lemma}\label{lem:compatible-real-forms}
Let $\G$ be a complex semisimple Lie group with a fixed Iwasawa factorization. Given any real form of $\G_0$,
there exists a real form $\G_0'$ conjugated to $\G_0$  such that $\G$ and $\G_0'$ have compatible
Iwasawa factorizations.
\end{lemma}
\begin{proof} Fix an arbitrary Iwasawa factorization $\G_0=\K_0\A_0\NN_0$.
We first promote it to an Iwasawa factorization of $\G$.
The Cartan decomposition $\theta:\G_0\to \G_0$ has a  anti-holomorphic extension to $\gg$, and there
it defines a Cartan decomposition of $\G$,
\[\gg=\kk\oplus \pp.\]
To obtain from $\aa_0$ a maximal abelian subalgebra of $\pp$,
define the centralizer of $\aa_0$ in $\kk_0$
\[\mm_0=\zz_{\kk_0}(\aa_0)\]
and take a maximal abelian subalgebra  $\tt_0$ of $\mm_0$. It follows that $\tt_0\oplus \aa_0$ is a Cartan subalgebra
of $\gg_0$. It also follows that $i\tt_0\oplus \aa_0$ is a maximal abelian subalgebra $\aa$ of $\pp$. According to
\cite[Chapter VI, Section 6]{Kn}
there exists a root ordering of $(\gg,\aa)$ such that for a root $\alpha$ with $\alpha|_{\aa_0}\neq 0$
one has that $\alpha$ is positive if and only if the restricted root $\alpha|_{\aa_0}$ is positive.
The outcome is an Iwasawa factorization $\G=\K\A\NN$ that
is compatible with the fixed Iwasawa factorization $\G_0=\K_0\A_0\NN_0$.

Finally, there exist an element $g\in \G$ such that the conjugation of $\G=\K\A\NN$ by $g$ produces the given
Iwasawa factorization of $\G$ \cite[Chapter VI, Section 6]{Kn}. Therefore $\G$ and its real form $g\G_0g^{-1}$
have compatible Iwasawa factorizations.
\end{proof}

\subsection{Chevalley big cells}
We describe  the relationship among  various Chevalley big cells associated 
with  $\G=\K\A\NN$ and $\G_0=\K_0\A_0\NN_0$,  a complex semisimple Lie group and a real form  with compatible Iwasawa factorizations. This is used to define the induced $(Y,Z)$ variables on $\cO_0$ in Theorem \ref{thm:relative}.

\medskip

The Chevalley big cell  of $\G_0$ sits inside the Chevalley big cell of $\G$,
\[\cC_0\subset \cC.\]
In general, $\cC_0$ is not a dense open subset of $\K_0\cap \cC$. We  describe a canonical enlargement of $\cC_0$ to an open dense subset of $\K_0\cap \cC$. This enlargement makes use of the Chevalley big cell of the complexification of the (possibly non connected) subgroup
 $\M_0=Z_{\K_0}(\aa_0)$, which integrates the subalgebra $\mm_0$ whose elements are the eigenvectors in $\kk_0$
associated with the eigenvalue 0 for the adjoint action. 

The complexification  of $\mm_0$ is
\[\mm=\ll_{\Im}\oplus \nn_{\Im}\oplus \tt_0\oplus i\tt_0,\]
where $\nn_\Im$ and $\ll_\Im$ denote the sum of positive and negative imaginary root spaces of $(\gg,\aa)$, respectively.
The compact group $\M_0$ admits an abstract complexification $\M$. From the properties of the complexification functor,
 the inclusion $\M_0\subset \K$ induces a holomorphic injective map from $\M$ into $\G$. Its
image is a closed subgroup that we identify with $\M$ and that integrates $\mm\subset \gg$. Moreover,  the Cartan decomposition on $\G$ induces a Cartan decomposition on $\M$, which at the Lie algebra level is
\[\mm_0=\kk\cap \mm,\quad i\mm_0=\pp\cap \mm.\]
Since $\tt_0$ is a Cartan subalgebra of $\mm_0$, $i\tt_0$ is a maximal abelian subalgebra of $i\mm_0$.
The roots of $(\mm,i\tt_0)$ are nothing but the imaginary roots of $(\gg,\aa)$, and they come with an ordering from
that of $(\gg,\aa)$. Therefore we have an Iwasawa factorization for $\M=\M_0\exp(i\tt_0)\NN_{\Im}$, where  $\NN_\Im\subset \G$ is the connected integration of $\nn_{\Im}$, and its corresponding Chevalley big cell 
\[ \cC_{\M_0}=\kappa_\M(\L_{\Im}),\quad\kappa_{\M}:\M\to \M_0,\]
where $\L_\Im\subset \G$ is the connected integrations of $\nn_\Im,\ll_\Im$. We denote by $\TT_0=Z_{\M_0}(i\tt_0)$.

\begin{lemma}\label{lem:Iwasawa-compatible} Let $\G=\K\A\NN$ and
 $\G_0=\K_0\A_0\NN_0$ be compatible Iwasawa factorizations, and let $\M=\M_0\exp(i\tt_0)\NN_{\Im}$ be the
 induced Iwasawa factorization of $\M$. Then
 \begin{enumerate}[(i)]
  \item  the product map  $\cC_0\times \cC_{\M_0}\to \cC_0\cC_{\M_0}$ is a diffeomorphism onto its image, a subset of $\K_0 \cap \cC$. More precisely,
\begin{equation*}
\cC_0\cC_{\M_0}=\kappa(\L_0\L_{\Im})\subset \cC,\quad \kappa:\G\to \K;
\end{equation*}
\item the subset
 $\cC_{\M_0}\cC_0\TT_0\L_0\A_0\subset \G_0$
is open and dense and equals $\cC_{\M_0}\cC_0\L_0Z_{\K_0}(\hh_0)$.
 \end{enumerate}

\end{lemma}
\begin{proof}
As for an element in $\G_0$ its Iwasawa factorizations in $\G_0$  and $\G$ agree, 
\[\cC_0=\kappa(\L_0)\subset \cC,\quad \kappa:\G\to \K.\]
Since
\[\M_0\subset \K,\quad \exp(i\tt_0)\subset \A,\quad \NN_{\Im}\subset \NN,\]
the Iwasawa factorizations in $\M$ and $\G$ also agree. Therefore
\[\cC_{\M_0}=\kappa(\L_\Im)\subset \cC.\]
Furthermore, if $l_1\in \L_0$, $l_2\in \L_{\Im}$ have Iwasawa factorizations
\[l_1=k_1u_1, \quad l_2=mu_2\]
in  $\G_0$ and $\M$, respectively, then,
since $\M_0$ normalizes $\U_0$, their product has Iwasawa factorization in $\G$
\[l_1l_2=k_1u_1mu_2=k_1mu_3u_2=k_1mu,\]
and this proves (i).

We claim that the multiplication map $\cC_0\times \M_0\to \cC_0\M_0$ is a diffeomorphism onto its image and that
 \[\cC_0\M_0\subset \K_0,\quad \M_0\cC_0\subset \K_0\]
are open and dense subsets. For $\G_0$ semisimple this
is proved in \cite[Section 2]{MT}. In general, $\G_0$ may be not connected, but is a reductive subgroup. Therefore, by Propositions 7.33 and 7.49 in \cite{Kn},
$\M_0$ intersects every connected component of $\K_0$, and thus the result also holds in this generality. This proves the claim.

We apply the previous result also to the reductive group $\M$ to conclude that
\[\cC_{\M_0}\TT_0\subset \M_0\]
is open and dense. Therefore
\[\cC_{\M_0}\TT_0\cC_0\subset \K_0\]
is an open and dense subset, and so $\cC_{\M_0}\TT_0\cC_0\L_0\A_0\subset \G_0$ also is.  We also claim that $\TT_0$ normalizes $\cC_0$. From this, \[\cC_{\M_0}\TT_0\cC_0=\cC_{\M_0}\cC_0\TT_0,\] and the equality
\[\cC_{\M_0}\cC_0\TT_0\L_0\A_0=\cC_{\M_0}\cC_0\L_0Z_{\G_0}(\hh_0).\]
is a consequence of
\[\TT_0\A_0=\Z_{\M_0}(i\tt_0)\Z_{\K_0}(\aa_0)=\Z_{\M_0}(\tt_0)\Z_{\K_0}(\aa_0)=\Z_{\K_0}(\hh_0),\]
and this proves (ii). Note that $\cC_{\M_0}\cC_0\subset \K_0\cap \cC$ is open and dense.

To prove the claim
we use that $\TT_0\subset \K_0$ and that  it normalizes $\L_0$ and $\U_0$, since $\TT_0\subset \M_0$. Therefore,
if $l\in \L_0$ has Iwasawa factorization
in $\G_0$, $l=ku$,
and then $tlt^{-1}\in \L_0$ with Iwasawa factorization
$tlt^{-1}=(tkt^{-1})(tut^{-1})$.
\end{proof}

\subsection{Weyl groups}

We   discuss the relationship among several Weyl groups associated with the data of compatible Iwasawa factorizations.
This is used to parametrize the elements in $\cO_0\cap \hh_0$ in Theorem \ref{thm:relative}.

\medskip

To  the complex Lie group with fixed Cartan subalgebra$(\G,\hh)$ one  associates the Weyl group of $(\G,\hh)$
\[\NN_\G(\hh)/\Z_\G(\hh).\]
If we regard $\G$ as a real Lie group, then to the fixed Iwasawa factorization one associates the Weyl group of $(\G,\aa)$, \[\W=\NN_\K(\aa)/\Z_\K(\aa).\]
These two quotient groups are canonically isomorphic, 
$
 \NN_\G(\hh)/\Z_\G(\hh)\to\W
$.

\medskip

To the real form $\G_0$ with fixed Iwasawa factorization one can associate both
\begin{itemize}
 \item  the canonically  isomorphic quotient groups
\[\NN_{\G_0}(\hh_0)/\Z_{\G_0}(\hh_0)\to \NN_{\K_0}(\hh_0)/\Z_{\K_0}(\hh_0),\]
\item and the Weyl group of $(\G_0,\aa_0)$
\[\W_0=\NN_{\K_0}(\aa_0)/\Z_{\K_0}(\aa_0).\]
\end{itemize}

The isomorphic quotient groups sit in Weyl group $\W$.
\begin{lemma}\label{lem:inclusion-Weyl} The inclusion $\NN_{\K_0}(\hh_0)\subset \NN_\K(\aa)$
 induces  the injection
 \[\NN_{\K_0}(\hh_0)/\Z_{\K_0}(\hh_0)\to \W.\]
\end{lemma}
\begin{proof} If $k\in \NN_{\K_0}(\hh_0)$, then it normalizes $\hh$ and thus it also normalizes its vector part $\aa$. This explains the inclusion of normalizers in the statement. For centralizers the argument is simpler, producing a well defined map between quotients.  The image of $k$ in $\W$ is trivial if and only if
$k\in \Z_\K(\aa)$. This implies that $k\in \Z_\K(\hh)$, and, because $k\in \K_0$, that
$k\in \Z_{\K_0}(\hh_0)$. Therefore the map is injective.
\end{proof}

We relate the two quotient groups for $\G_0$ by means of the Weyl group of $(\M_0,i\tt_0)$,
\[\W_{\M_0}=\NN_{\M_0}(i\tt_0)/\Z_{\M_0}(i\tt_0).\]

\begin{lemma}\label{lem: Weyl short exact} The  inclusion $\NN_{\K_0}(\hh_0)\to \NN_{\K_0}(\aa_0)$ induces the short exact sequence
 \[1\to \W_{\M_0}\to \NN_{\K_0}(\hh_0)/\Z_{\K_0}(\hh_0)\to \W_0\to 1.\]
\end{lemma}
\begin{proof}
The inclusion between normalizers is again a consequence of the invariance of toroidal and vector components.
To show that the map induced on quotients is onto we use that $\tt_0$ is a Cartan subalgebra of $\M_0=Z_{\K_0}(\aa_0)$.
If $k\in \NN_{\K_0}(\aa_0)$ is such that $\tt_0\neq \tt_0^k$, then since all Cartan subalgebras of a compact Lie algebra are conjugated, we can always find $m\in \M_0$ such that
$\tt_0^{gm}=\tt_0$.

By construction, the kernel of the map induced on quotients are those elements in $\NN_{\K_0}(\hh_0)/\Z_{\K_0}(\hh_0)$ any of whose representatives $k\in \NN_{\K_0}(\hh_0)$ has fixed-point set containing $\aa_0$. But this means that $k\in\M_0$ as well,
\[k\in \NN_{\K_0}(\hh_0)\cap \M_0=\NN_{\M_0}(i\tt_0).\]
Therefore the kernel  equals
$\NN_{\M_0}(i\tt_0)/\Z_{\M_0}(i\tt_0)=\W_{\M_0}$.
\end{proof}

\end{document}